\documentclass[11pt]{article}
\usepackage{authblk}

\usepackage[latin1]{inputenc}
\usepackage[T1]{fontenc}
\usepackage{amsmath}
\usepackage{amsfonts}
\usepackage{amssymb}

\usepackage[normalem]{ulem} 

\usepackage{amsmath,xcolor,ulem}
\usepackage{amsfonts}
\usepackage{amssymb,todonotes}
\usepackage{amsthm}
\usepackage{graphicx}
\usepackage{marginnote}
\newcommand{\dd}{\text{d}}

\usepackage[colorlinks=true, allcolors=blue]{hyperref}

\usepackage[top=2.8cm,bottom=2.8cm,left=2.5cm,right=2.5cm]{geometry}
\usepackage{float}
\usepackage[algoruled]{algorithm2e}
\usepackage[font=footnotesize,width=\linewidth]{caption}

\newcommand{\R}{\mathbb R}

\newcommand{\yd}[0]{y_\textrm{data}}
\newcommand{\J}[0]{\mathcal J}

\newcommand{\Uad}{\mathcal U_\text{ad}}
\newcommand{\e}{\epsilon}

\DeclareMathOperator{\chirp}{chirp}
\DeclareMathOperator{\squa}{square}

\newtheorem{theorem}{Theorem}

\newtheorem{lemma}[theorem]{Lemma}
\newtheorem{problem}[theorem]{Problem}
\newtheorem*{remark}{Remark}

\title{Data-driven adjoint-based identification of port-Hamiltonian systems in time domain}

\author{Michael G\"unther\footnote{Research Group Applied and Computational Mathematics, \href{mailto:guenther@uni-wuppertal.de}{guenther@uni-wuppertal.de}} ,
Birgit Jacob\footnote{Corresponding author, Research Group Functional Analysis, \href{mailto:bjacob@uni-wuppertal.de}{bjacob@uni-wuppertal.de}} ,
Claudia Totzeck\footnote{Research Group Optimization, \href{mailto:totzeck@uni-wuppertal.de}{totzeck@uni-wuppertal.de}  }}
\affil{IMACM, School of Mathematics and Natural Sciences, \\ University of Wuppertal, Germany}

\begin{document}
\maketitle

\begin{tikzpicture}[remember picture,overlay]
	\node[anchor=north east,inner sep=20pt] at (current page.north east)
	{\includegraphics[scale=0.2]{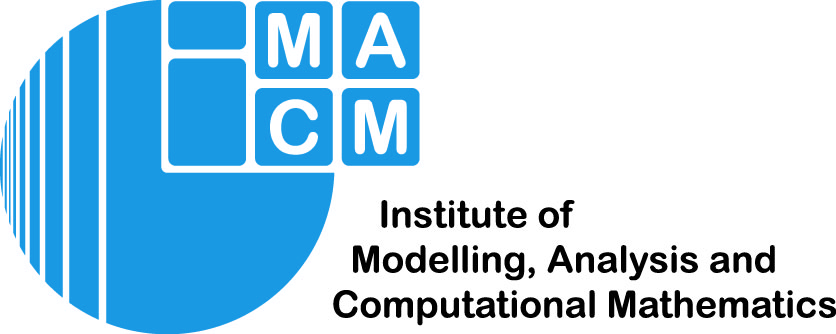}};
\end{tikzpicture}

\begin{abstract}We present a gradient-based identification algorithm to identify the system matrices of a linear port-Hamiltonian system from given input-output time data. Aiming for a direct structure-preserving approach, we employ techniques from optimal control with ordinary differential equations and define a constrained optimization problem. The input-to-state stability is discussed which is the key step towards the existence of optimal controls. Further, we derive the first-order optimality system taking into account the port-Hamiltonian structure. Indeed, the proposed method preserves the skew-symmetry and positive (semi)-definiteness of the system matrices throughout the optimization iterations.  Numerical results with perturbed and unperturbed synthetic data, as well as an example from the PHS benchmark collection~\cite{PHSbenchmarks} demonstrate the feasibility of the approach.
\end{abstract}

\begin{minipage}{0.9\linewidth}
\medskip

\end{minipage}

\section{Introduction}
In structure-preserving modelling of coupled dynamical systems the port-Hamil\-tonian framework allows for constructing overall port-Hamiltonian systems (PHS) provided that (a) all subsystems are PHS and (b) a power-conserving interconnection between the input and outputs of the subsystems is provided~\cite{MeMo19, vanDerSchaft06,EbMS07,DuinMacc09}.
 In realistic applications this approach reaches its limits if  for a specific subsystem, either no knowledge that would allow the definition of a physics-based PHS is available, or one is forced to use user-specified simulation packages with no information of the intrinsic dynamics, and thus only the input-output characteristics are available. 
 In both cases a remedy is to generate input-output data either by physical measurements or evaluation of the simulation package, and derive a PHS surrogate that fits these input-output data best. This PHS surrogate can then be used to model the subsystem, and overall leads to a coupled PHS with structure-preserving properties.

There is an extensive literature on system identification methods, see for example \cite{Lj99}. Here we focus on the identification of port-Hamiltonian systems. 
Port-Hamiltonian realizations are a special class of realizations of passive systems. More precisely, in \cite{ChGeHi23} it is shown that passive systems have a port-Hamiltonian realization. For a detailed survey on passive systems we refer to \cite{wi72}. Data-driven port-Hamiltonian realizations of dynamic systems from input-output data have been investigated using various approaches. The Loewner framework has been used to construct a port-Hamiltonian realization from frequency domain data in \cite{AnLeIo,benner2020identification}. Another frequency domain approach is proposed in \cite{SCH}, where a parametrization of the class of PHS is used to permit the usage of unconstrained optimization solvers during identification. In \cite{ChGB22} the frequency response data are inferred using the time-domain input-output data and then frequency domain methods are used to construct a port-Hamiltonian realization.  Further, in \cite{cherifi2019optimization} a best-fit linear state-space model is derived from time-domain data and then, in a post-processing step, a nearest port-Hamiltonian realization is inferred. 
Based on  input, state and output time-domain data and a given quadratic Hamiltonian the authors of \cite{MoNiUn22} construct a port-Hamiltonian realization using dynamic mode decompositions.

 We propose a direct time-domain approach for constructing a best-fit PHS model in one step using input-output data. Techniques from optimal control with ordinary differential equations and a constrained optimization problem are employed. We derive the first-order optimality system taking into account the port-Hamiltonian structure. The proposed method preserves the skew-symmetry and positive (semi)-definiteness of the system matrices throughout the optimization iterations. Our approach does not generate frequency data as an intermediate step and data of the state variable are not needed. We remark, that our method is overdetermined because we identify all system parameters including the quadratic Hamiltonian. We account for this by showing different test cases in the numerical section. If the Hamiltonian is known it is possible to use our approach to derive a best-fit PHS with a given Hamiltonian.

The paper is organized as follows: in the next section, we define the identification problem of computing the best-fit of a PHS to given input-output data, prove the continuous dependence of the state on the input which is the key step to obtain the existence of solutions to the identification problem. The adjoint equation and optimality conditions for all system matrices, i.e., the positive definite scaling matrix $Q$, the fixed-rank semi-definite dissipation matrix $R$, the skew-symmetric system matrix $J$, the input matrix $B$ and the initial value, are derived in section three. The gradient-descent algorithm and numerical schemes for the identification process are stated in chapter four. Various numerical tests are provided in chapter five: in particular, we provide several examples for identification problems in the deterministic case and we test the robustness in case of noisy data. The numerical examples are rounded off with an investigation of the model as model order reduction approach, here we use a single mass-spring-damper chain example from the PHS benchmark collection~\cite{PHSbenchmarks}. We conclude with an outlook.

\section{Identification problem}
For a given input $u \colon [0,T] \rightarrow \R^m,$ a data set $\yd \colon [0,T] \rightarrow \R^m$ which can consist of continuous data on the interval $[0,T]$ for $T>0$ or continuous interpolation of measurements at discrete time points and reference values for the identification $w_\text{ref},$ we consider the identification problem with cost functional
\begin{equation}\label{eq:costfunctional}
\J(y,w) = \frac{1}{2} \int_0^T |y(t) - \yd(t) |^2 \dd t + \frac{\lambda}{2} | w-w_\text{ref} |^2, \qquad \lambda \ge 0,
\end{equation}
subject to the state constraint
\begin{subequations}\label{eq:state}
\begin{align}
    \frac{d}{dt} x &= (J-R) Qx + B u, \qquad x(0)=x_0 \in \R^n, \label{eq:ODE}\\
    y &= B^\top Q x. \label{eq:output}
\end{align}
\end{subequations}
where $w = (J,R,Q,B,x_0)$ contains the matrices and initial data to be identified having the properties
\begin{equation}\label{eq:matrixproperties}
 J,R,Q \in \R^{n\times n}, \qquad J^\top = -J, \qquad R \succcurlyeq 0, \qquad Q \succ 0, \qquad B \in \R^{n\times m}.
\end{equation}
If no information on reference data is available, we recommend to set $\lambda=0$.
Note that only output data and for $\lambda>0$ reference values  are considered in the cost functional. Hence, the dimension $n$ of the internal state is unknown. We assume in the following that $n$ is chosen a priori and remark that a deliberate small choice can also be interpreted as a model order reduction for the internal state. For notational convenience, we define
\[
\J_1(y) := \frac{1}{2} \int_0^T |y(t) - \yd(t) |^2 \dd t, \qquad \J_2(w) := \frac{\lambda}{2} | w-w_\text{ref} |^2.
\]
In applications often $m\le n$ hence we obtain only partial information about the internal state $x$. In particular, we have no information about the initial condition $x_0.$ This is the reason why we include the initial data in the identification process for $w.$ Although we use techniques from optimal control theory, we may refer to $w$ as parameters in the following and introduce the parameter space
\[ W = \R^{n\times n} \times \R^{n\times n} \times \R^{n\times n} \times \R^{n\times m} \times \R^n
\]
with set of admissible parameters as
\[
\Uad \subseteq \{ w = (J,R,Q,B,x_0) \in W \colon J^\top = -J,R \succcurlyeq 0,  Q \succ 0  \}.
\]
Moreover, we define the output space $Y=H^1((0,T),\R^m)$.

For later reference we define the identification task at hand:
\begin{problem}\label{problem}
We seek to find system matrices and initial data, $w=(J,R,Q,B,x_0),$ which solve the problem
\begin{equation}\label{opt_prob}\tag{IdP}
\min\limits_{(y,w) \in Y \times \Uad } \J(y,w) \quad \text{subject to} \quad \eqref{eq:state}.
\end{equation}
\end{problem}

In \cite{guenther2022structure}, we proposed a sensitivity-based approach to identify the system matrices of a port-Hamiltonian system and argued that this is a valid ansatz for small-scale systems, as it requires to solve auxilliary problems for each basis element of the matrix spaces. Here, we derive an adjoint-based approach. Before we begin the derivation of the algorithm, we analyze the well-posedness of the identification problem. First, we prove that the state solution depends continuously on the data.

\begin{lemma}\label{lem:contdep}
For every $w \in \Uad$ and $u \in C([0,T],\R^m)$ the state equation \eqref{eq:state} admits a unique solution $x\in C^1([0,T],\R^n)$. Moreover, the solution depends continuously on the data, in more detail, for $w, w' \in \Uad$ and corresponding solutions $x,x'$ there exists a constant $C>0$ such that
\[
\| x - x' \|_{H^1((0,T),\R^n)} \le C \|w-w'\|_{\Uad}.
\]
\end{lemma}
\begin{proof}
For given $w\in\Uad$ and $u \in C([0,T],\R^m)$ the existence of a unique solution to \eqref{eq:state} follows by standard ODE theory \cite{teschl}. For the second statement we estimate $\| x - x' \|_{L_2((0,T),\R^n)}$ and $\| \frac{\dd}{\dd t}x - \frac{\dd}{\dd t} x' \|_{L_2((0,T),\R^n)}$ separately. 

We obtain
\begin{align*}
| x(t) - x'(t) |^2 &\le 2| x_0 - x_0'|^2 + 2T\int_0^t | (J-R)Q x(s) - (J' -R')Q' x'(s) |^2 \dd s \\
&\le C_1 | w - w' |^2 + \int_0^t C_2 | x(s) - x'(s) |^2 \dd s
\end{align*}
An application of Gronwall inequality yields
\[
|x(t) - x'(t)|^2 \le C_3 |w-w'|^2.
\]
Integration over $[0,T]$ we obtain
\begin{equation}\label{eq:est1}
\| x - x' \|_{L^2((0,T),\R^n)}^2 \le C_4 |w-w'|^2.
\end{equation}
Moreover, we obtain for
\begin{align*}
|\frac{\dd}{\dd t}  \big(x(t) - x'(t)\big)|^2 &= C_5 |w-w'|^2 + C_6\int_0^t |\frac{\dd}{\dd s}  \big(x(s) - x'(s)\big)|^2 \dd s. 
\end{align*}
Again, Gronwall and integration over $[0,T]$ yields
\begin{equation}\label{eq:est2}
\| \frac{\dd}{\dd t} \big( x(t) - x'(t) \big) \|_{L^2((0,T),\R^n)}^2 \le C_7 |w-w'|^2
\end{equation}
Adding \eqref{eq:est1} and \eqref{eq:est2} and taking square root yield the desired result.
\end{proof}

The well-posedness of the state equation allows us to introduce the \textit{solution operator} $$S \colon \Uad \rightarrow H^1((0,T),\R^n), \qquad S(w) = x.$$ Moreover, we define the reduced cost functional 
$$\hat \J(w) = \frac{1}{2} \int_0^T | B^\top QS(w)(t) - \yd(t) |^2 \dd t + \frac{\lambda}{2} |w-w_\text{ref}|^2. $$
Both will be helpful in the proof of the well-posedness result of the identification problem and also play a major role in the derivation of the gradient-descent algorithm later on. 

The next ingredient for the well-posedness result is the weak convergence of the state operator
\begin{equation}\label{eq:stateOperator}
e \colon H^1((0,T),\R^n) \times \Uad \rightarrow L^2((0,T),\R^n) \times \R^n, \qquad (x,w) \mapsto \begin{pmatrix} \frac{\dd}{\dd t} x - (J-R)Qx - Bu \\
x(0) - x_0 \end{pmatrix}.
\end{equation}
\begin{lemma}\label{lem:weakcontstate}
The state operator $e$ defined in \eqref{eq:stateOperator} is weakly continuous.
\end{lemma}
\begin{proof}
We consider $\{x_n\}_n \subset H^1((0,T),\R^n)$ with $x_n \rightharpoonup x$ and $\{w_n\}_n \subset \Uad$ with $w_n \rightharpoonup w.$ Note that $\Uad$ is finite dimensional, thus weak and strong convergence coincide. Then for $\varphi\in L^2((0,T),\R^n)$ we obtain 
\begin{align*}
&\int_0^T \left( \frac{\dd}{\dd t} x_n - (J_n-R_n)Q_n x_n(t) - B_n u(t)\right) \cdot \varphi(t) dt - \int_0^T \left( \frac{\dd}{\dd t} x - (J-R)Q x(t) - B_n u(t)\right) \cdot\varphi(t) dt  \\
&= \int_0^T \left( \frac{\dd}{\dd t} x_n - \frac{\dd}{\dd t}x \right) \cdot\varphi (t)\, \dd t - \int_0^T((J_n - J) - (R_n-R))Q_n x_n(t) \, \cdot\varphi(t) \dd t \\
&\qquad - \int_0^T (J-R)(Q_n-Q) x_n(t) \cdot \varphi(t) + (J-R)Q(x_n(t)-x(t)) \cdot \varphi(t) + (B_n - B) u(t) \cdot \varphi(t) \dd t.
\end{align*}
By the weak convergence of $x_n$ and $w_n$  we conclude that all terms tends to zero as $n\rightarrow \infty$. Note that weak convergence of $\{x_n\}_n$ implies boundedness of the sequence $\{x_n\}_n$. Moreover, by the embedding $H^1((0,T),\R^n) \hookrightarrow C([0,T],\R^n)$ we conclude that $x_n(0) \rightarrow x(0)$ and further we obtain $(x_0)_n \rightarrow x_0$ by the convergence of $w_n.$ Altogether, this yields the desired result.
\end{proof}

\begin{theorem}\label{opt_wellposedness}
Let $\lambda > 0$ or $\Uad$ bounded. Then \eqref{opt_prob} admits a global minimum.
\end{theorem}
\begin{proof}
We follow the lines of \cite{troelzsch,hinze2008optimization} and consider a minimizing sequence $\{w_n\}_n,$ that means $\lim\limits_{n\rightarrow \infty}\hat \J(w_n) = \inf\limits_{w\in\Uad} \hat \J(w).$

If $\lambda>0$, $\hat \J(w)$ is coercive w.r.t.~$w$ and hence the minimization sequence is bounded. In the other case the minimizing sequence is bounded by assumption. As $\Uad$ is a subset of a reflexive, finite-dimensional space, we can extract a convergent subsequence $\{w_{n_k}\}_k$ with $w_{n_k} \rightarrow w^*.$ Note that $w=0$ implies $z\equiv0,$ hence the continuous dependence on the data shown in Lemma~\ref{lem:contdep} implies the boundedness of $\{S(w_{n_k})\}_k \subset H^1((0,T),\R^n)$ and we can extract a weakly convergent subsequence $S(w_{n_{k_\ell}}) \rightharpoonup x$ in $H^1((0,T),\R^n).$ The weak continuity of the state operator shown in Lemma~\ref{lem:weakcontstate} yields
\[
e(S(w_{n_{k_\ell}}),w_{n_{k_\ell}}) \rightharpoonup e(x,w) 
\]
and the weak lower semicontinuity of the norm allows us to estimate 
\[
\| e(x,w) \|_{L^2((0,T),\R^n) \times \R^n} \le \liminf\limits_{\ell \rightarrow \infty} \| e(S(w_{n_{k_\ell}}),w_{n_{k_\ell}}) \|_{L^2((0,T),\R^n)\times \R^n} = 0,
\]
which proves that $S(w)=x.$ We note that the weak lower semicontinuity of $\J$ implies the weak lower semicontinuity of $\hat\J,$ hence
\[
\hat\J(w) \le \lim\limits_{\ell\rightarrow\infty} \hat\J(w_{n_{k_\ell}}) = \inf\limits_{w\in\Uad} \hat\J(w).
\]
This proves the existence of a minimizer.
\end{proof}

\begin{remark}
We want to emphasize that even though the state equation is linear in $x$, it is non-linear in $w$. Thus in general, the identification problem is nonconvex and  we cannot expect to obtain a uniqueness result for \eqref{opt_prob}.
\end{remark}

Here, we aim for a general approach, also feasible for high-dimensional systems, and therefore discuss an adjoint-based approach in the following. A challenge in this context is the structure of the system matrices \eqref{eq:matrixproperties} which we aim to preserve in each step of the identification process. To this end, we employ techniques from optimization on manifolds \cite{absil2008optimization} to compute the respective retractions for the system matrices to preserve their structure in each iteration. To the authors' knowledge, there is no explicit formula for the retraction in the space of semi-definite matrices. Hence, we use the flat metric for $R$ and as the space of skew-symmetric matrices as well as $\R^{n\times m}$ and $\R^n$ are vector spaces, we are naturally in the flat case for $J$, $B$ and $x_0,$ respectively. 

One of our objectives is the derivation of a gradient-based algorithm for the identification. We derive the adjoint-based gradient descent scheme step-wise, i.e., we discuss the computation of the first-order optimality system and then show that the optimality condition of this system coincides with the gradient of the reduced cost functional.

\section{First-order optimality system}
The main goal of this section is the derivation of the gradient of the reduced cost functional $\hat\J$. We follow a standard approach from optimization with differential equations, see for example \cite{hinze2008optimization} and use the Lagrangian of the constrained dynamics, which is given by
\[
L(x,w,p) = \J(x,w) - \langle e(x,w), p \rangle,
\]
where $p$ denotes the Lagrange multiplier or adjoint variable. The first-order optimality system is derived by solving $dL=0$. Hence we compute
\begin{align*}
d_x L(x,w,p)[h_x] &= d_x\J(x,w)[h_x] - \int_0^T d_xe(x,w)[h_x] \cdot p \dd t - h_x(0) \cdot p_0 \\
&= \int_0^T QB(y-\yd) \cdot h_x + \frac{d}{dt} p + Q(J^\top - R^\top) p \cdot h_x \dd t - h_x(0) \cdot p_0.
\end{align*}
As $d_x L(x,w,p)[h_x] = 0$ must hold for arbitrary $h_x$ we can identify the adjoint equation as
\begin{equation}\label{eq:adjoint}
-\frac{\dd}{\dd t} p = Q(J^\top - R^\top) p + QB(y- y_\mathrm{data}), \qquad p(T)=0.
\end{equation}
The derivative $d_p L(x,w,p)[h_p] = 0$ yields the state equation and the optimality condition is derived as
\begin{equation*}
d_w L(x,w,p)[h_w] = d_w\J(x,w)[h_w] - \int_0^T d_we(x,w)[h_w] \cdot p \dd t.
\end{equation*}
For simplicity we split the derivation into the different matrices, to obtain
\begin{align*}
0&=\nabla_Q \J_2(w) + \int_0^T x(t) \otimes \big((J^\top - R^\top)p(t)\big) + x(t) \otimes (B(y(t)-\yd(t))) \dd t, \\
0&=\nabla_J \J_2(w) + \int_0^T p(t) \otimes (Qx(t)) \dd t, \\
0&=\nabla_R \J_2(w) - \int_0^T p(t) \otimes (Qx(t)) \dd t, \\
0&=\nabla_B \J_2(w) + \int_0^T \left( p(t) \otimes u(t) + (Qx(t)) \otimes (y(t)-\yd(t))  \right) \dd t, \\
0&=\nabla_{x_0} \J_2(w) + p(0).
\end{align*}

Altogether, we obtain the following result.
\begin{theorem}\label{thm:optconditions}
A stationary point $\bar w = (\bar Q, \bar R, \bar J, \bar B, \bar x_0)$ of Problem~\ref{problem} satisfies the  optimality condition
\begin{subequations}\label{eq:gradient}
\begin{align}
&\nabla_Q \J_2(\bar w) + \int_0^T \bar x\otimes \bar p  \,dt = 0, \\
&\nabla_R \J_2(\bar w) -\int_0^T \left( \bar x \otimes \bar R \bar p + \bar p \otimes \bar R \bar x \right)  \, dt  =0, \\
&\nabla_J \J_2(\bar w) + \int_0^T \bar p \otimes \bar x \,\dd t = 0, \\
&\nabla_{x_0} \J_2(\bar w) +\bar p(0) = 0, \\
&\nabla_{B} \J_2(\bar w) +\int_0^T \left( \bar Q\bar x \otimes (\bar y-y_\textrm{data}) +  \bar Q^\top \bar B(\bar y-y_\mathrm{data}) \otimes \int_0^t u(s) ds \right) \, dt =0,
\end{align}
\end{subequations}
where $\bar p$ satisfies the adjoint equation
\begin{equation}
-\frac{\dd}{\dd t} \bar p = \bar Q(\bar J^\top - \bar R^\top) \bar p + \bar Q\bar B(\bar y- y_\mathrm{data}), \qquad \bar p(T)=0.
\end{equation}
and $\bar x$ the state equation with output $\bar y$ given by
\begin{align*}
\frac{d}{dt} \bar x &= (\bar J-\bar R) \bar Q\bar x + \bar Bu, \qquad \bar x(0)=\bar x_0, \\
\bar y &= \bar B^\top \bar Q \bar x,
\end{align*}
where $a \otimes b$ denotes the dyadic product $a b^\top$ for $a,b\in\R^n$.
\end{theorem}

To stay on the respective manifolds in each iteration of the identification procedure, we use the retractions from manopt \cite{manopt} in the implementation.

\subsection{Relationship to the gradient of the reduced cost functional}
We are now well-prepared to identify the gradient of our identification problem. We discuss the case including the estimation of the initial state, if this is needless the gradient update with respect to $z_0$ can be neglected. We have already seen that the solution operator allows us to define the reduced cost function which which leads us to an reformulation of \eqref{opt_prob} as unconstrained optimization problem, i.e., the state constraint is treated implicitly. The gradient we derive in the following is the one of the reduced cost functional. In fact, the whole gradient-descent algorithm will only vary $w$ and thereby vary the state implicitly.

Let $w = (Q,R,J,B,z_0)$ be the parameters or, to be more precise, the matrices and initial data to be identified. We denote by $\langle \cdot, \cdot \rangle_{W^*,W} \colon W^*,W \rightarrow \R$ the dual pairing of $W$ and its dual $W^*.$ Moreover, $\mathcal A^*$ denotes the adjoint of the operator $\mathcal A.$  To identify  the gradient of the reduced cost functional we first note that the state operator yields
\[
 0= d_y e(S(w),w) S'(w)[h]+ d_w e(S(w),w)[h] \quad \Rightarrow \quad S'(w)[h] = -d_y e(S(w),w)^{-1} d_w e(S(w),w)[h].
\]
This allows us to find
\begin{align*}
\langle \hat\J'(w) , h \rangle_{W^*,W} &= \langle d_y \J(y,w), B^\top Q S'(w)[h] \rangle_{H^{-1},H^1} + \langle d_w \J(y,w), h\rangle_{W^*,W} \\
&= \langle d_w \J(y,w) -d_w e(S(w),w)^* d_y e(S(w),w)^{-*} [QB\,d_y \J(y,w)], h \rangle_{W^*,W}
\end{align*}
Since $W$ is a Hilbert space and $h$ was chosen arbitrarily, Riesz representation theorem allows us to identify
\begin{equation}\label{eq:riesz}
(\nabla \J(w),h)_W =  \langle d_w \J(y,w) -d_w e(S(w),w)^* d_y e(S(w),w)^{-*} [QB\,d_y \J(y,w)], h \rangle_{W^*,W}.
\end{equation}

In the previous subsection we computed the adjoint equation already and the gradient components corresponding $(Q,R,J,B,x_0)$, respectively. 

We emphasize that the computation \eqref{eq:riesz} shows that the left-hand side of system \eqref{eq:gradient} is in fact the gradient of the reduced cost functional at $\bar w$. We will use this information in the gradient-descent algorithm proposed in the following section.

\section{Algorithm and numerical schemes}
An optimal solution has to satisfy the conditions in Theorem~\ref{thm:optconditions} all at once. However, due to the forward/backward structure of coupled state and adjoint equation  it is in general difficult to solve the system directly. This is the reason why we propose an iterative approach based on the gradient \eqref{eq:gradient}. In fact, for an initial guess of systems matrices and initial condition $w$ we solve the state equation \eqref{eq:state} and obtain $S(w),$ using this information we solve the corresponding adjoint equation \eqref{eq:adjoint}. The information of the state and adjoint solutions allows us to evaluate the gradient at $w$ with $-\nabla \hat\J(w)$ we update our initial guess for the second iteration using a step size obtained by Armijo rule with initial stepsize $\sigma_0$ \cite{hinze2008optimization}. The following iterates are based on a conjugate-gradient (CG) update as proposed in \cite{sato2021riemannian}, in more detail, the search direction is given by
\begin{equation}\label{eq:cg-direction}
\eta_{k+1} = -\nabla \J(w_{k+1}) + \beta_{k+1} s_k \mathcal T^{(k)}(\eta_k)
\end{equation}
with scaling parameter
$$s_k = \min\left\{1, \frac{\| \eta_k \|_{w_k}}{\| \mathcal T^{(k)}(\eta_k) \|_{w_{k+1}}}\right\} $$
and transport map $\mathcal T^{(k)}$ that transports the tangent vector at $w_k$ to a similar tangent at $w_{k+1}$. This transport is required in order to have well-defined CG-steps in the manifold setting, see \cite{sato2021riemannian} for more details. The transport maps for the respective manifolds implemented in the manopt package \cite{manopt} are used. In the implementation, we reset the CG-step in every $5$.~iteration.

The iteration is stopped if an appropriate stopping criterion is fulfilled. As stopping criterion we check  the relative cost: First, we can stop if the gradient vanishes numerically, i.e., $\| \nabla \hat\J(w) \| < \e$ for $0< \e \ll 1.$ Another option is to stop the iteration, if the update from one to the next gradient step $\|\nabla \hat\J(w^k) - \nabla\hat \J(w^{k+1})\| < \e$ is sufficiently small. Additionally, we can impose a maximal number of iterations to stop the iteration. We summarize the steps in Algorihm~\ref{alg:gradientdescent}.

\begin{minipage}{0.9\linewidth}
\begin{algorithm}[H]
\caption{Conjugate-gradient algorithm for the identification process}\label{alg:gradientdescent}
\KwData{initial guess $w,$ algorithmic parameters, stopping criterion\;}
\KwResult{optimized matrices and initial data $\bar w$}
\SetCustomAlgoRuledWidth{1cm}
\SetAlgoLined
1) solve state equation $\rightarrow S(w)$\;
2) solve adjoint equation $\rightarrow p$\;
3) evaluate gradient \eqref{eq:gradient}$\rightarrow \nabla\J(w)$\;
4) compute the CG-direction \eqref{eq:cg-direction} $\rightarrow \eta$\; 
5) find admissible step size with Armijo rule $\rightarrow \sigma$\;
6) update control $w \mapsto w + \sigma \eta$\;
7) \textbf{if} stopping criterion is not fulfilled $\rightarrow$ go to 1)\; 
$\quad$\textbf{else} return optimal control\;
\end{algorithm}
\end{minipage}

For the numerical results presented in the next section, we use the following algorithmic parameters unless explicitly stated otherwise: maximal budget of $100$ CG-iterations, initial step size for Armijo rule $1000$. The time step size of the explicit Euler discretizations of the state and adjoint ODE, respectively, is $0.001$ and the upper bound of the time interval considered for the ODEs is $T=1$. The integrals in the gradient expressions are computed with an simple left-endpoint approximation.

In order to avoid numerical errors to interfere with the structure of the matrices, we check for skew-symmetry or symmetry, respectively, and if needed overwrite the new iterate for $J$ by $J \mapsto \frac{1}{2}(J - J^\top)$ and analogous $Q$ by $Q \mapsto \frac{1}{2}(Q + Q^\top)$ and   $R$ by $R \mapsto \frac{1}{2}(R + R^\top).$ As we assume to have no knowledge about the systems, we use no reference values for the tests and set $\J_2(w) \equiv 0.$ 

\begin{figure}[ht!]
    \centering
    \includegraphics[scale=0.45]{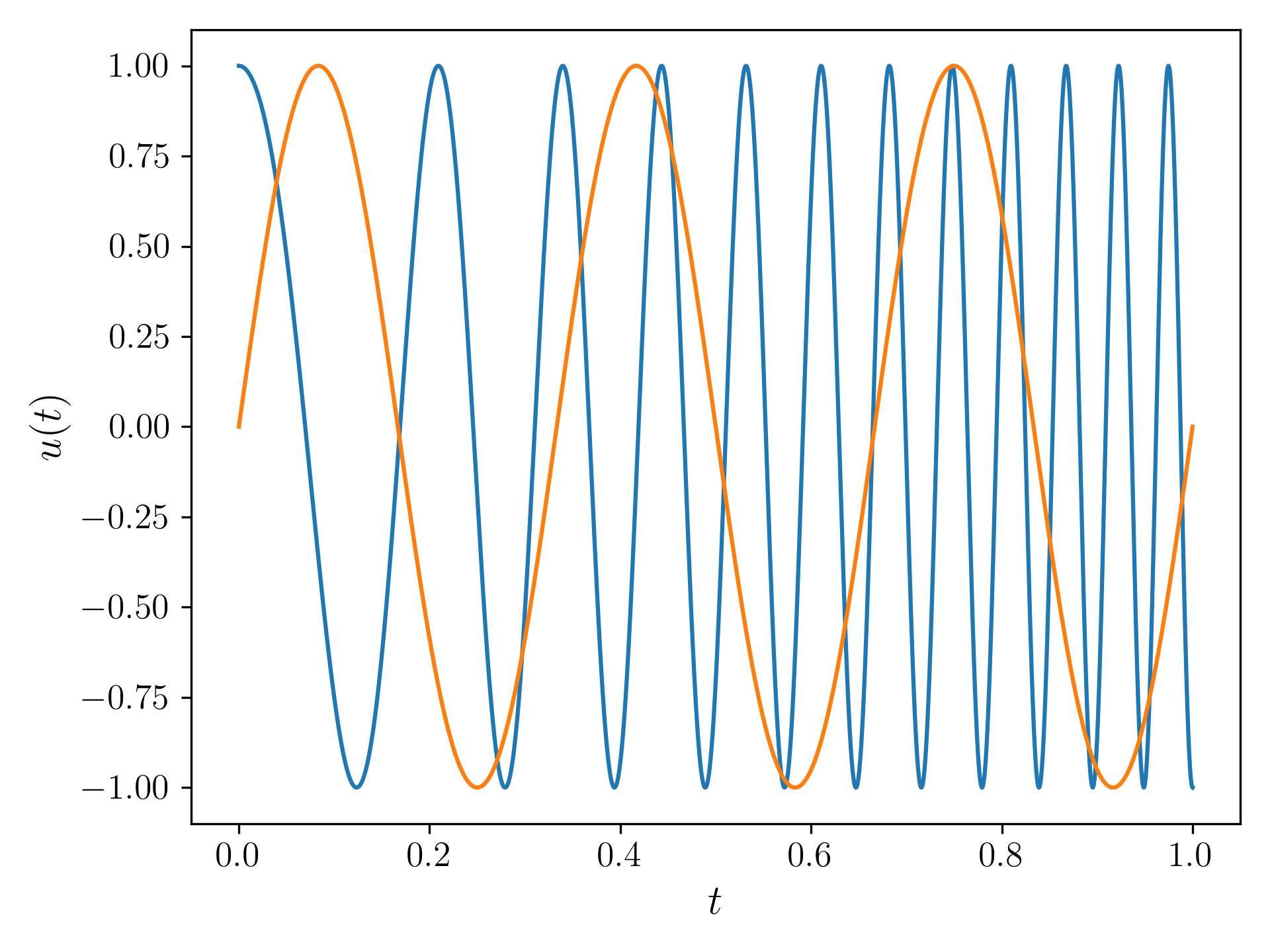}
    \includegraphics[scale=0.45]{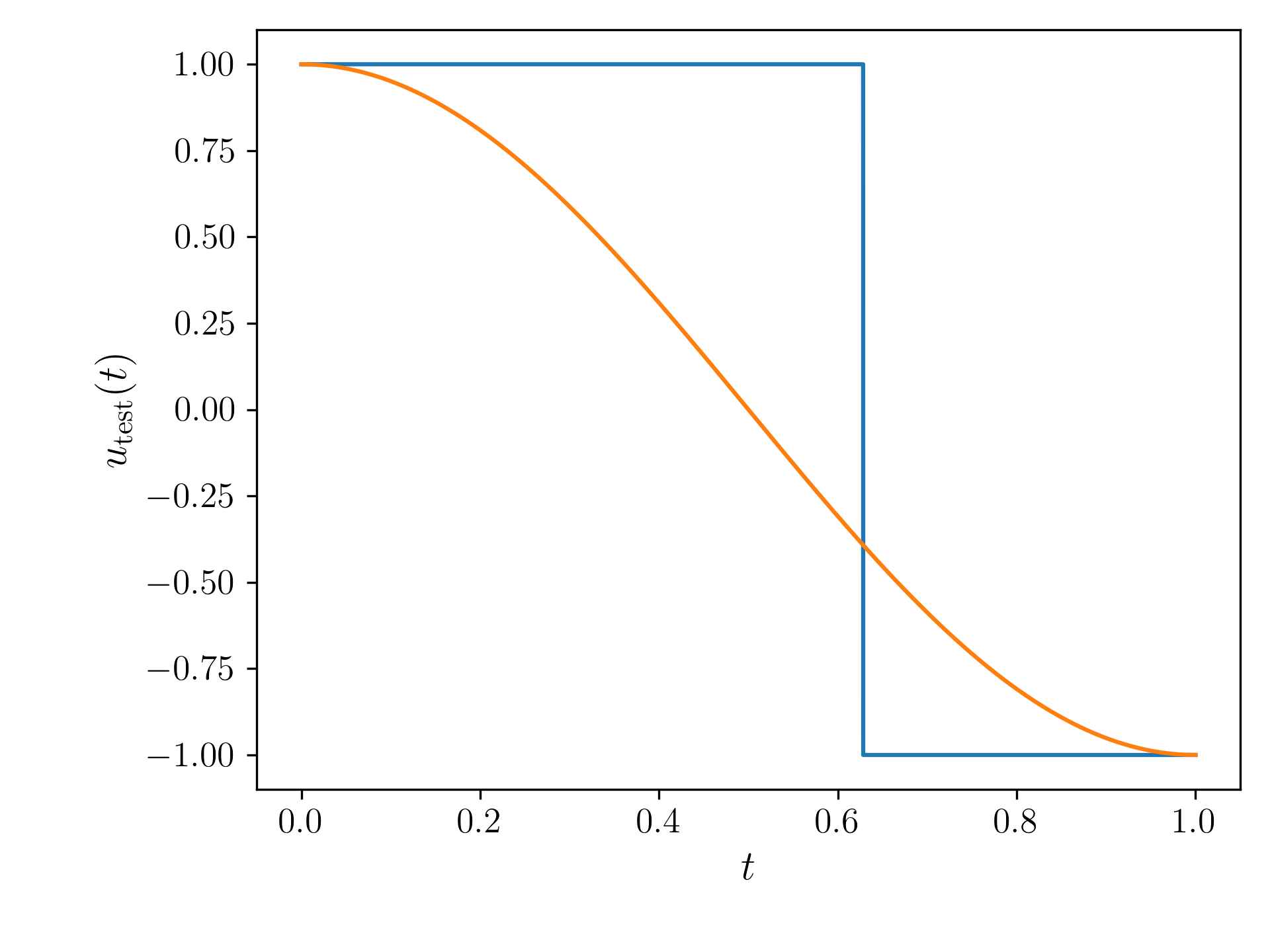}
    \caption{Signal $u$ for the identification (left) and signal $u_\text{test}$ for cross validation (right).}
    \label{fig:inputs}
\end{figure}
Moreover, we set $m=2$ for all the numerical tests and use the inputs (see Figure~\ref{fig:inputs})
\begin{equation*}
u \colon [0,T] \rightarrow \R^2, \qquad u(t) = \begin{pmatrix} \chirp_{3,20}(t) \\ \sin(2\pi f_0 t) \end{pmatrix},
\end{equation*}
where $\chirp$ denotes the chirp-signal from \texttt{scipy.signal} toolbox with frequency $3$ at time $t=0$ and frequency $20$ at time $t=1$ and $f_0 = 3$.

We initialize $B$ with an $\text{diag}(1,3)$ in the first $m\times m$ entries and zero everywhere else. The matrix $Q$ representing the Hamiltonian is initialized as the $n\times n$ identity matrix. For $R$ we set the first $2 \times 2$ zero and to identity and zero everywhere else. Of course it is nontrivial to find a good initial guess for $R$ in the dynamics as in particular the number of dissipative elements is unknown. Furthermore, we assume that the system is in equilibrium if no input applies, therefore we refrain to estimate the initial state $x_0$ and set $x_0 \equiv 0$ in most of the examples. To generate a feasible initial $J$ we draw an $n\times n$ matrix with independent uniformly distributed entries in $[-1,1],$ take the upper part of this matrix (starting with the first off-diagonal) and fill the lower part of the  matrix by the negative of the transpose of the upper part.  

To check the robustness of the identified system, we perform a cross validation using the test signal
\begin{equation*}
u_\text{test} \colon [0,T] \rightarrow \R^2, \qquad u_\text{test}(t) = \begin{pmatrix} \squa_{0.1}(t) \\ \cos(\pi t) \end{pmatrix},
\end{equation*}
where $\squa_{0.1}$ is the square signal from the \texttt{scipy.signal} toolbox with duty parameter $0.1$.

\section{Proof of concept by numerical tests}
To underline the feasibility of the proposed approach, we show three test cases. First, we generate synthetic data and show that Algorithm~\ref{alg:gradientdescent} is able to find system matrices that approximate the data set. Then we repeat the test with a randomly perturbed data set. Our third test case tests the performace of a model order reduction with the proposed algorithm, here we apply the single mass-spring-damper chain from the PHS benchmark collection \cite{PHSbenchmarks}. The code is publicly available on github\footnote{\href{https://github.com/ctotzeck/PHScalibration}{https://github.com/ctotzeck/PHScalibration}}. 

\subsection{Gradient test}
To validate our implementation, we perform a gradient check using the finite differences in some arbitrary direction $h$ with elements in the respective tangent spaces. As we are in the manifold setting, the finite difference reads
\[
\frac{\hat \J(R_w(th)) - \hat \J(w)}{t} \approx (\nabla \hat\J(w),h)_w
\]
with retraction $R_w$ at $w$. Since $w=(J,R,Q,B,x_0)$ the retractions and also the inner product $(\cdot,\cdot)_w$ are defined element-wise. For details we refer to the code \texttt{gradient\_check.py} on github or \cite{manopt}. Also the directions in $h$ depend on the respective manifolds of $J,R,Q,B$ and $x_0$. 

We first test the finite differences where only one of the parameters $J,R,Q,B,x_0$ is varied, then we vary all parameters a once. For both test, the relative difference between the directional derivative computed with inner product and the (sum) of the finite differences is below $2$\%. In more detail, we obtain
\[
\frac{| (\nabla \hat\J(w),h)_w - \sum_K \frac{\hat \J(R_w(th_K)) - \hat \J(w)}{t} |}{|\sum_K \frac{\hat \J(R_w(th_K)) - \hat \J(w)}{t}|} = 0.0104
\]
and
\[
\frac{| (\nabla \hat\J(w),h)_w - \frac{\hat \J(R_w(th)) - \hat \J(w)}{t} |}{|\sum_K \frac{\hat \J(R_w(th)) - \hat \J(w)}{t}|} = 0.00484.
\]

\subsection{Synthetic data (deterministic)}
In the following we show results for a synthetic data set with $n=5$ and $m=2$. First, we fix $Q$ and $x_0$ and fit $J,R$ and $B$. Then we fix all matrices and just fit the initial value of the internal dynamics $x_0$, which is known in practise. 

\subsubsection*{Identification of $J,R, B$ and $Q$}
In this section we fit the system matrices $J,R,B$ and $Q$. In the first experiment, we neglegt $Q$ and in the second test, we include $Q$ and penalize deviations of $Q$ from the identify matrix.

In the first test, we assume to have no information about the system, hence we set $\lambda = 0$ and we fit the matrices $J,R,B$. In particular, we fix a positive-definite matrix $Q$ which is unchanged during the identification. This is justified as $Q$ always multiplies $x$ in the dynamics, hence learning $J,R,B$ will cover all necessary  information of the system. The reference output is generated using random matrices with appropriate structure.  
\begin{figure}[ht!]
    \centering
    \begin{minipage}{0.49\textwidth}
    \includegraphics[scale=0.5]{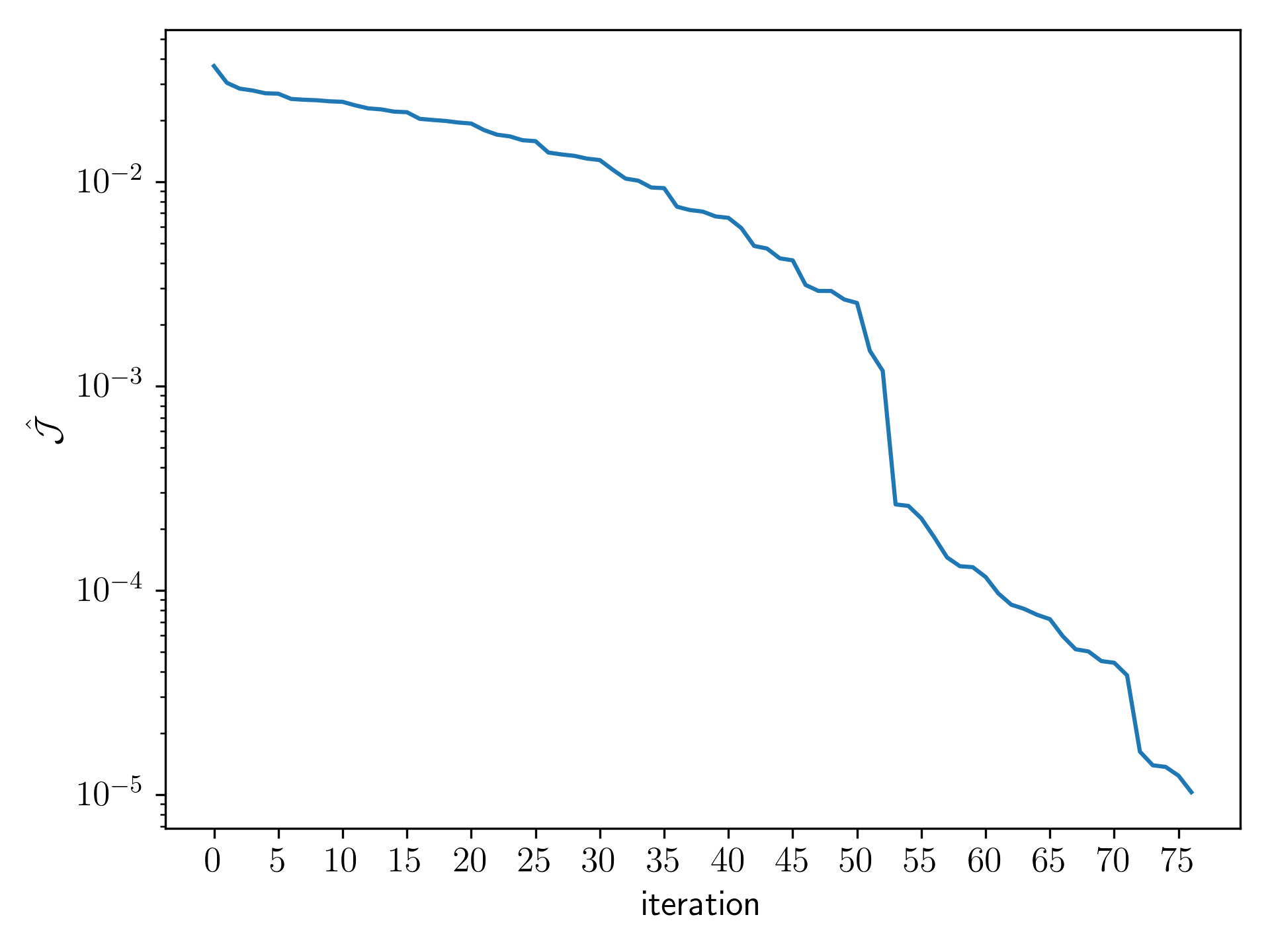}
    \end{minipage}
    \begin{minipage}{0.49\textwidth}
    \begin{tabular}{l | c }
    cost in first iteration & $3.68e^{-2}$  \\
    cost in last iteration & $1.03e^{-5}$ \\
    $\|y - y_\text{ref}\|_\infty$ in first iteration & $7.06e^{-1}$ \\
     $\|y - y_\text{ref}\|_\infty$ in last iteration & $1.01e^{-2}$ \\
    \end{tabular}
    \end{minipage}
    \caption{Left: Evolution of the cost over the iterations of the identification in $\log$-plot. Right: cost values before / after the identification and difference of fitted output and reference output in the maximum norm.}
    \label{fig:cost}
\end{figure}

\begin{figure}[ht!]
    \centering
    \includegraphics[scale=0.4]{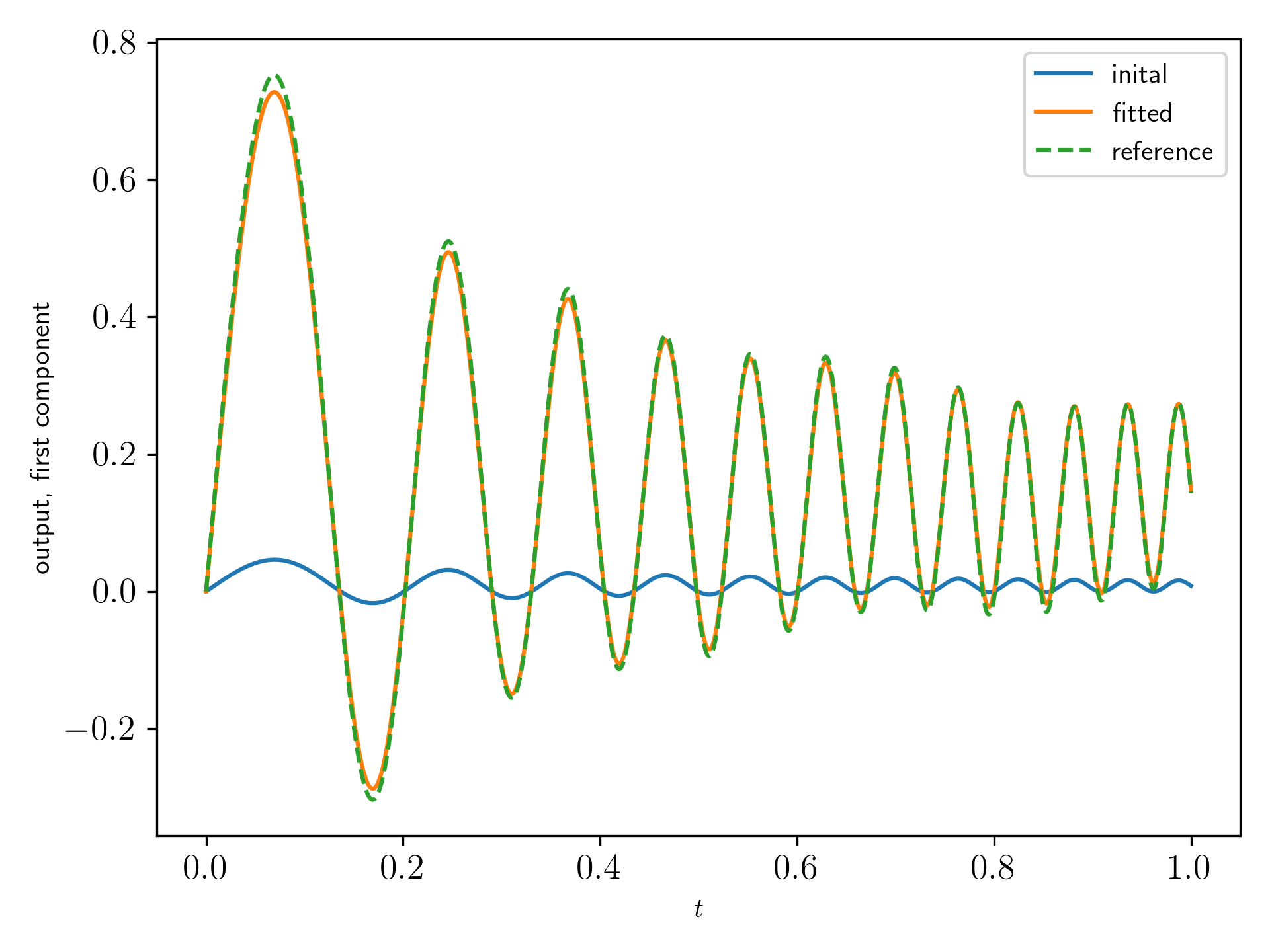}
    \includegraphics[scale=0.4]{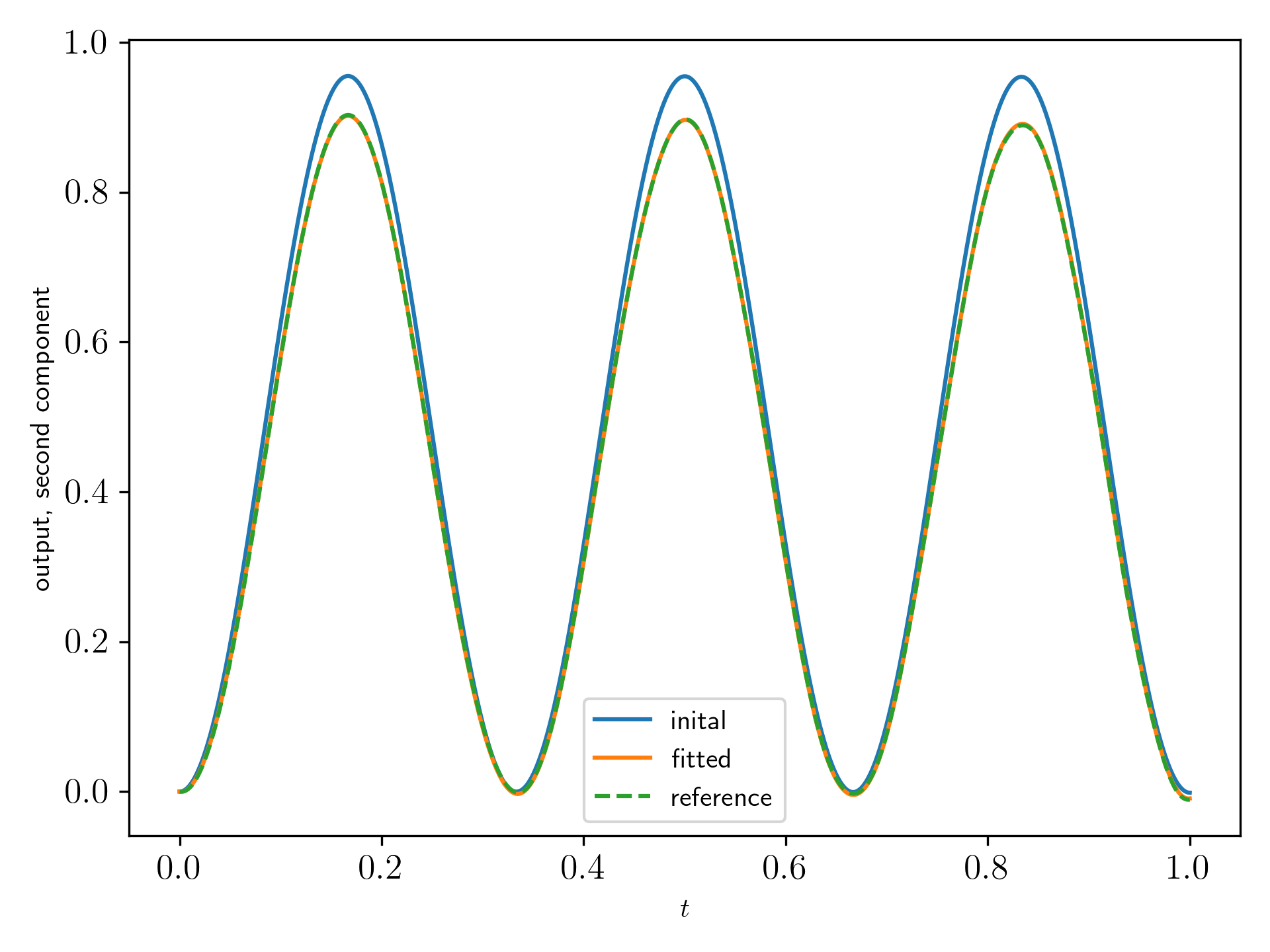}
    \caption{Left: First component of the output for the before the identification (initial), after the identification (fitted) and for the reference matrices (reference) used for the generation of the synthetic data. Right: Second component of the output for the before the identification (initial), after the identification (fitted) and for the reference matrices (reference)used for the generation of the synthetic data.}
    \label{fig:dataoutput_det}
\end{figure}

The algorithm terminated after the maximal number of iterations as can be seen in Figure~\ref{fig:cost}, where we show the evolution of the cost functional on the left-hand side and some additional information of the difference between the fitted and reference output in the maximum norm. Note that we cannot expect higher accuracy, since the error in the cost is of the order of the accuracy of the Euler scheme that we use to solve the state and adjoint systems. The reduction of the cost functional is above $99\%$. 

In Figure~\ref{fig:dataoutput_det} we show the output obtained with the initial matrices before the identification in blue, the orange graph shows the output of the fitted system and the green-dashed line shows the output obtained with the reference data. As the error values above already indicated the fitted output matches the reference output very well.

We close the numerical analysis of the identification of $J,R$ and $B$ with a cross validation test. The input signal $u_\text{test}$ is now applied to the fitted and reference systems and we compare the corresponding output. Figure~\ref{fig:robustness_det} shows the results of this cross validation graphically. In numbers, we obtain for the L2-difference $\| y - y_\text{test}\|_{L^2((0,1);\mathbb R^2)} = 1.63e^{-1}$ and for maximum norm $\| y - y_\text{test}\|_{\infty} = 2.96e^{-1}$. 
\begin{figure}[ht!]
    \centering
    \includegraphics[scale=0.45]{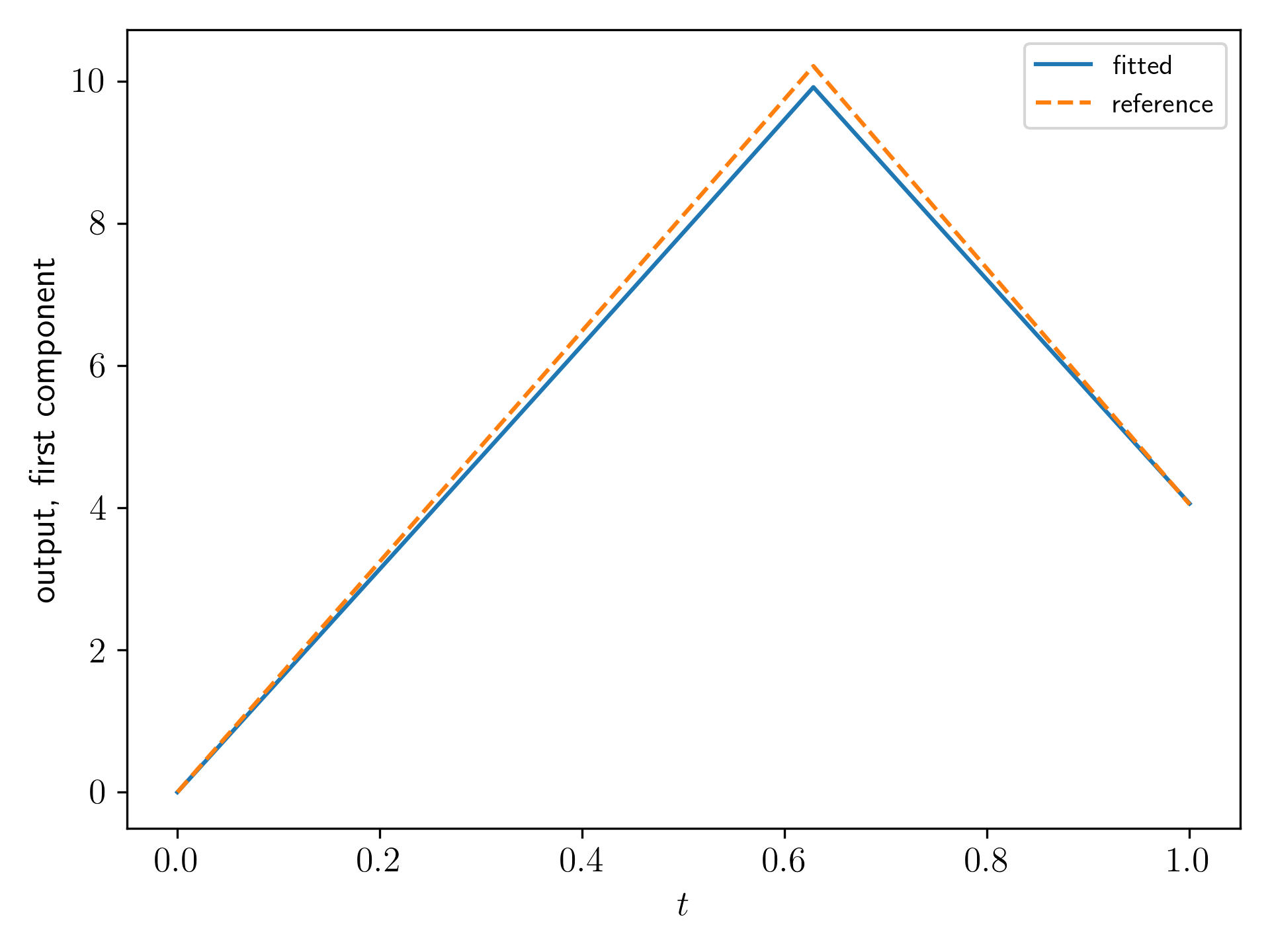}
    \includegraphics[scale=0.45]{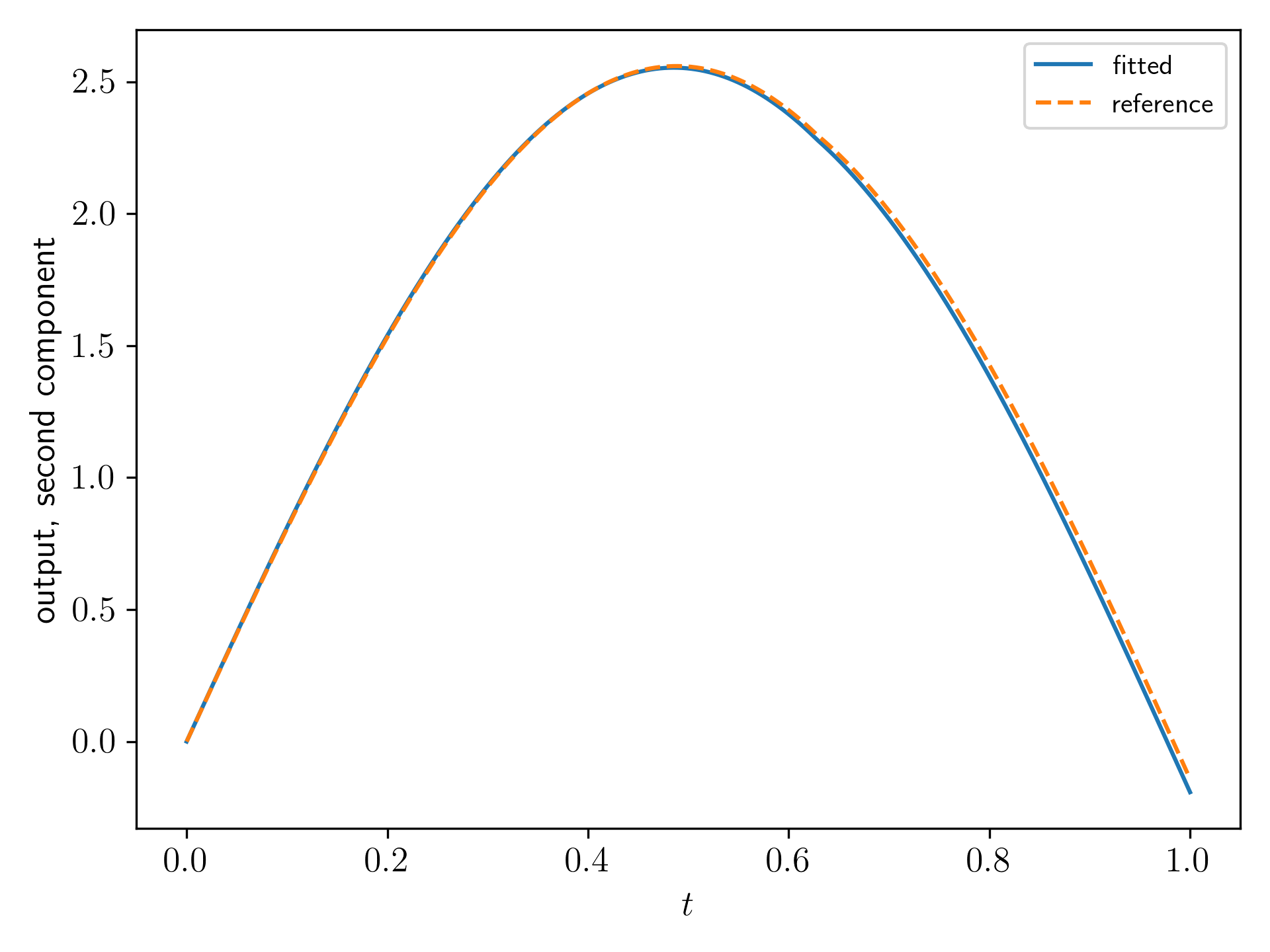}
    \caption{Left: First component of the output obtained in the cross validation. Right: Second component of the output obtained in the cross validation.}
    \label{fig:robustness_det}
\end{figure}

In the following, we learn $Q$ in addition to the other parameters learned above. To stabilize the identification process, we choose $\lambda = 0.001$, and the regularization term $\J_2(Q;I)= \frac12\|\log(Q) \|_F^2$ with derivative $Q^{-1}\log(Q)$ in the cost functional. On the one hand, this additional term stabilizes the identification process on the other hand, it tries to keep $Q$ close to the identify, which is away from the true value \texttt{\text{diag}(1,2,3,4,5)}.

\begin{figure}[ht!]
    \centering
    \begin{minipage}{0.49\textwidth}
    \includegraphics[scale=0.5]{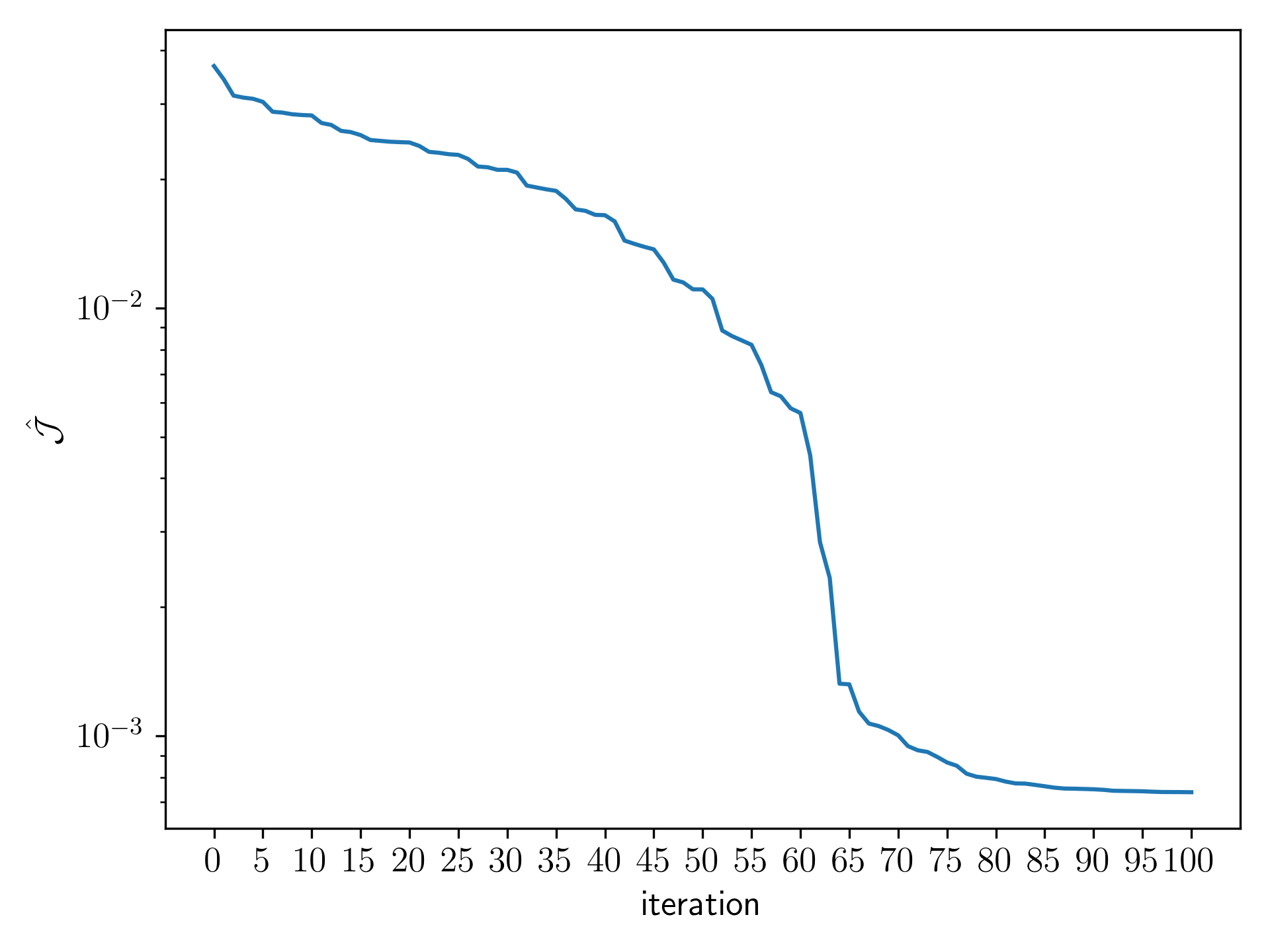}
    \end{minipage}
    \begin{minipage}{0.49\textwidth}
    \begin{tabular}{l | c }
    cost in first iteration & $3.68e^{-2}$  \\
    cost in last iteration & $7.48e^{-4}$ \\
    $\|y - y_\text{ref}\|_\infty$ in first iteration & $7.06e^{-1}$ \\
     $\|y - y_\text{ref}\|_\infty$ in last iteration & $8.08e^{-3}$ \\
    \end{tabular}
    \end{minipage}
    \caption{Left: Evolution of the cost over the iterations of the identification in $\log$-plot. Right: cost values before / after the identification and difference of fitted output and reference output in the maximum norm.}
    \label{fig:cost_lam}
\end{figure}

\begin{figure}[ht!]
    \centering
    \includegraphics[scale=0.4]{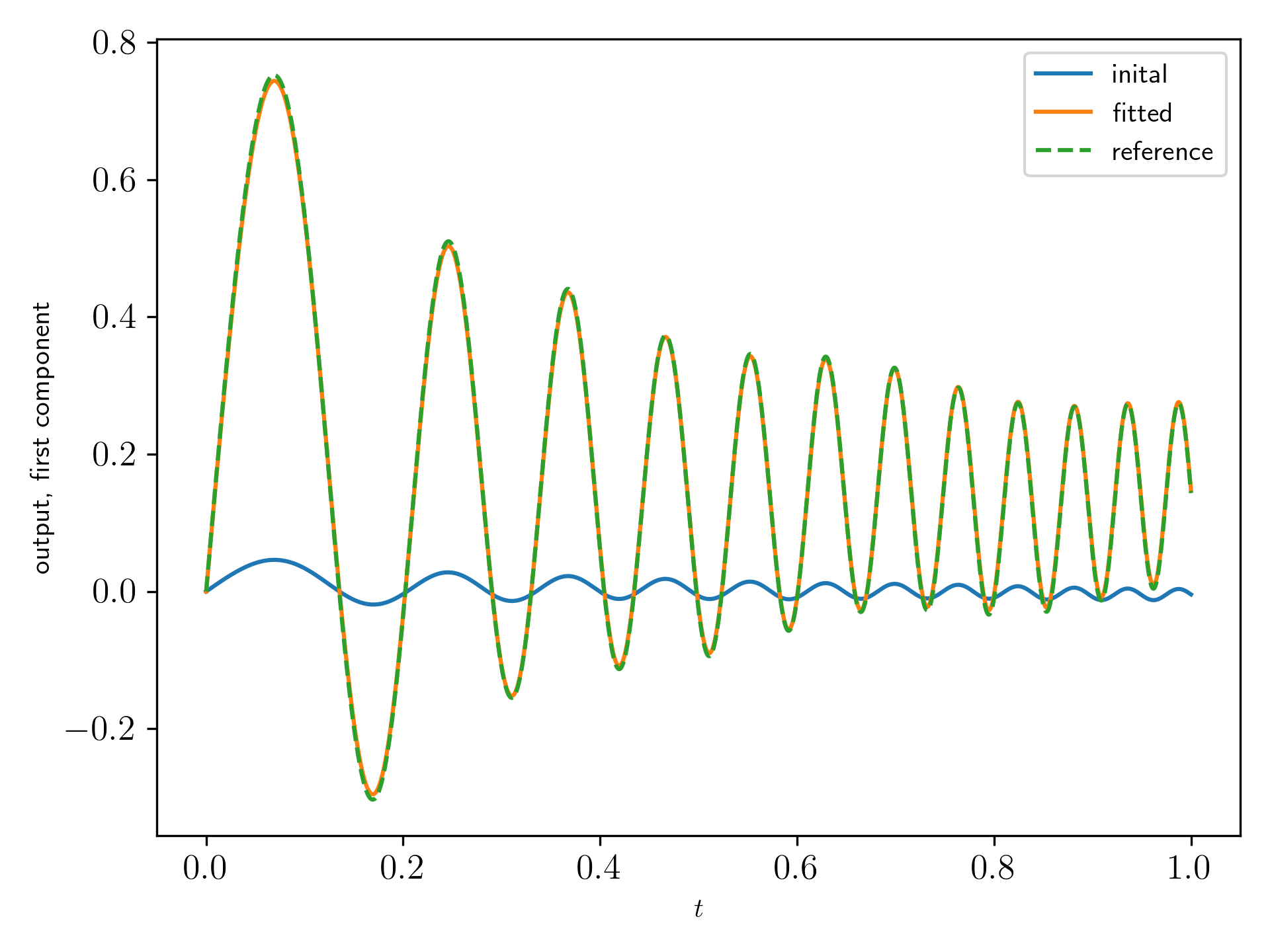}
    \includegraphics[scale=0.4]{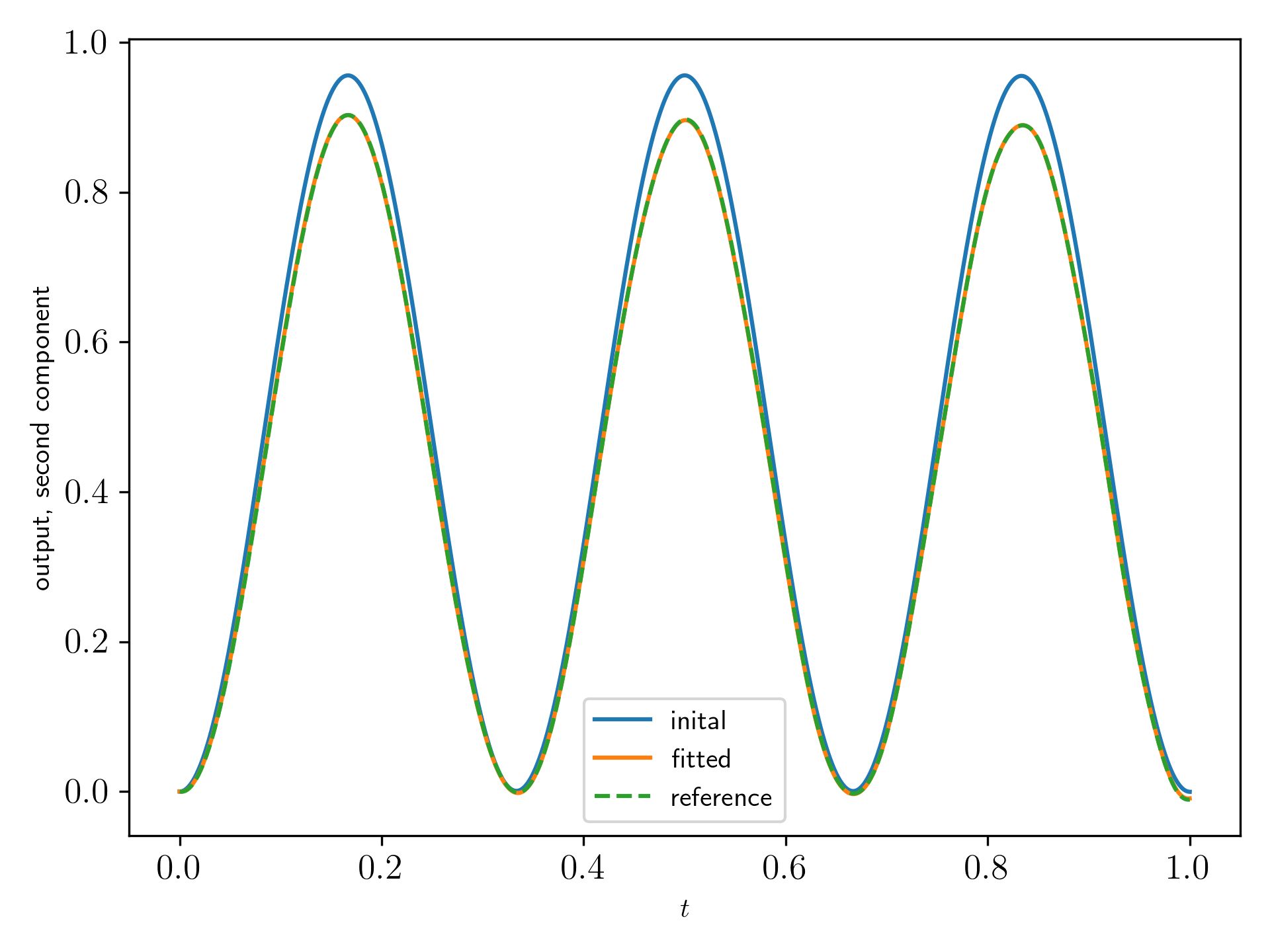}
    \caption{Left: First component of the output for the before the identification (initial), after the identification (fitted) and for the reference matrices (reference) used for the generation of the synthetic data. Right: Second component of the output for the before the identification (initial), after the identification (fitted) and for the reference matrices (reference)used for the generation of the synthetic data.}
    \label{fig:dataoutput_det_lam}
\end{figure}

The algorithm terminated after the maximal number of iterations as can be seen in Figure~\ref{fig:cost_lam}, where we show the evolution of the cost functional on the left-hand side and some additional information of the difference between the fitted and reference output in the maximum norm. The reduction of the cost functional is above $97\%$. 

In Figure~\ref{fig:dataoutput_det_lam} we show the output obtained with the initial matrices before the identification in blue, the orange graph shows the output of the fitted system and the green-dashed line shows the output obtained with the reference data. As the error values above already indicated the fitted output matches the reference output very well.

We close the numerical analysis of the identification of $J,R, B$ and $Q$ with a cross validation test. The input signal $u_\text{test}$ is now applied to the fitted and reference systems and we compare the corresponding output. Figure~\ref{fig:robustness_det_lam} shows the results of this cross validation graphically. In numbers, we obtain for the L2-difference $\| y - y_\text{test}\|_{L^2((0,1);\mathbb R^2)} = 2.41e^{-1}$ and for maximum norm $\| y - y_\text{test}\|_{\infty} = 4.16e^{-1}$. 
\begin{figure}[ht!]
    \centering
    \includegraphics[scale=0.45]{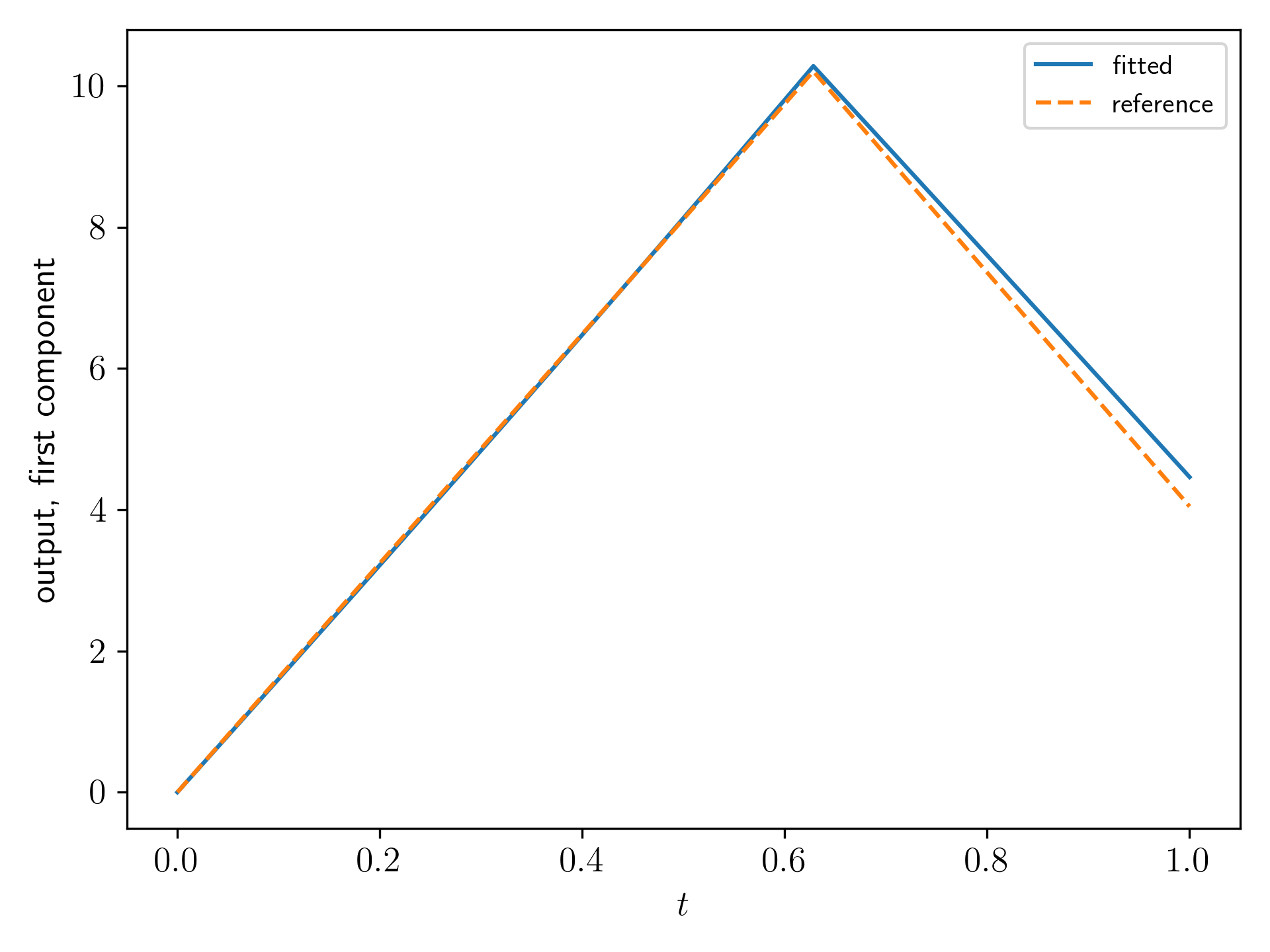}
    \includegraphics[scale=0.45]{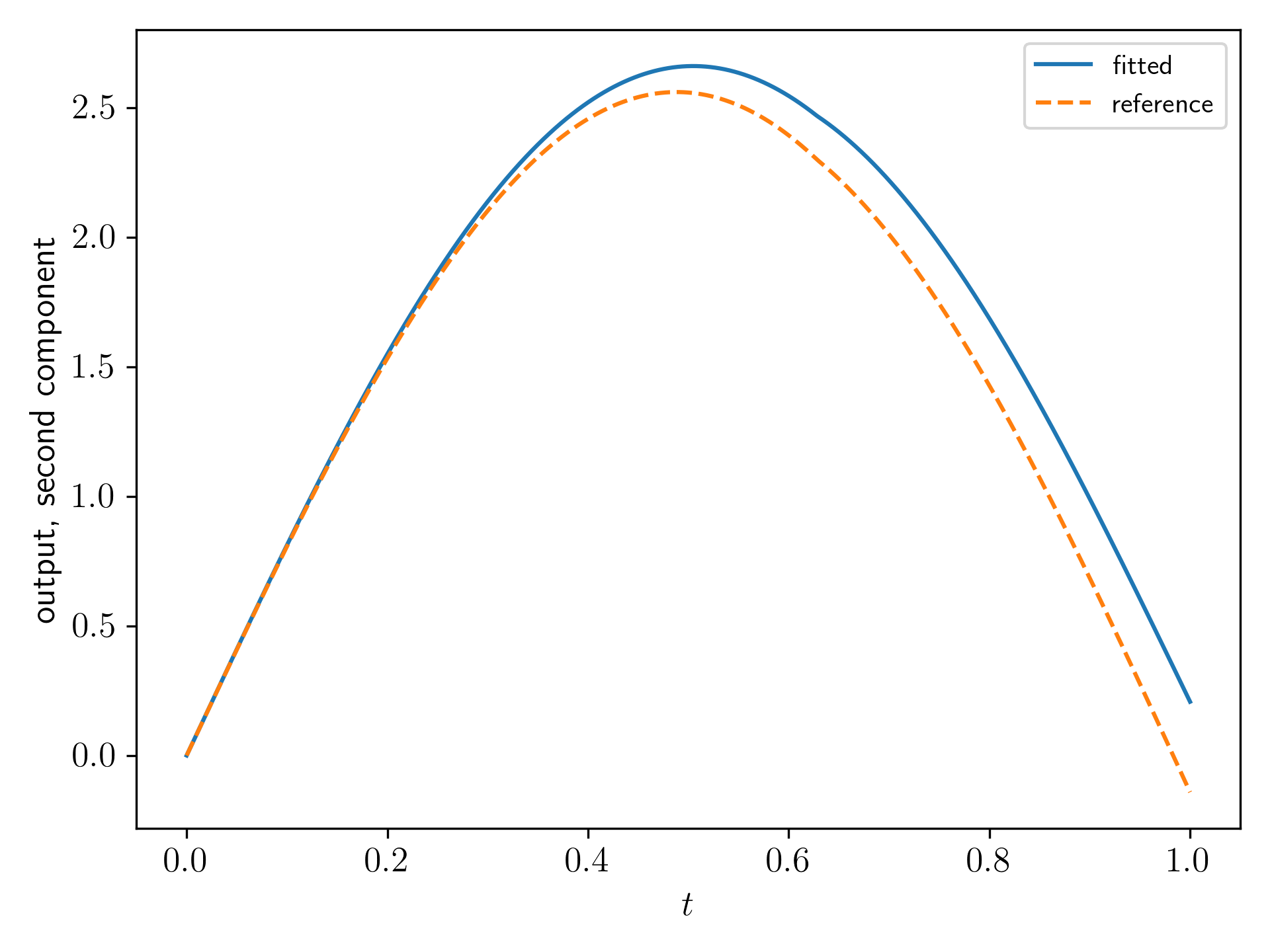}
    \caption{Left: First component of the output obtained in the cross validation. Right: Second component of the output obtained in the cross validation.}
    \label{fig:robustness_det_lam}
\end{figure}
As expected, the penalization of $Q$ stabilizes the parameter fitting process, but as the true value of $Q$ is in general unknown, it may cause an error by keeping $Q$ close to the identity.

\subsubsection*{Identification of $x_0$}
Now, we assume to know the system matrices $J,Q,R$ and $B$ and we are interested to fit the initial value of the internal state $x_0$, which is in practise unknown. The parameters are chosen as before.

\begin{figure}[ht!]
    \centering
    \begin{minipage}{0.49\textwidth}
    \includegraphics[scale=0.5]{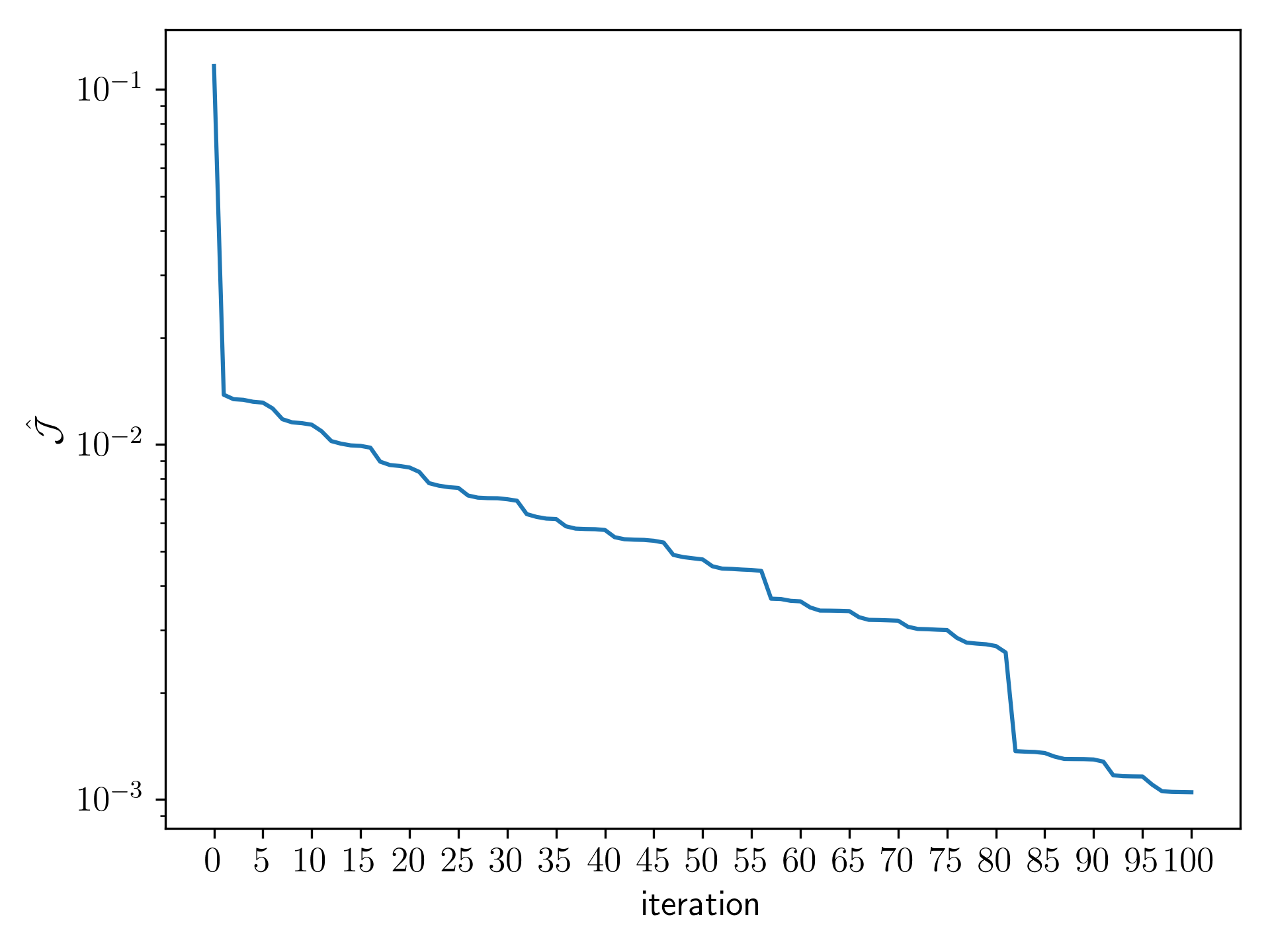}
    \end{minipage}
    \begin{minipage}{0.49\textwidth}
    \begin{tabular}{l | c }
    cost in first iteration & $1.16e^{-1}$  \\
    cost in last iteration & $1.05e^{-3}$ \\
    $\|y - y_\text{ref}\|_\infty$ in first iteration & $6.38e^{-1}$ \\
     $\|y - y_\text{ref}\|_\infty$ in last iteration & $8.75e^{-2}$ \\
    \end{tabular}
    \end{minipage}
    \caption{Left: Evolution of the cost over the iterations of the identification in $\log$-plot. Right: cost values before / after the identification and difference of fitted output and reference output in the maximum norm.}
    \label{fig:costX}
\end{figure}

\begin{figure}[ht!]
    \centering
    \includegraphics[scale=0.4]{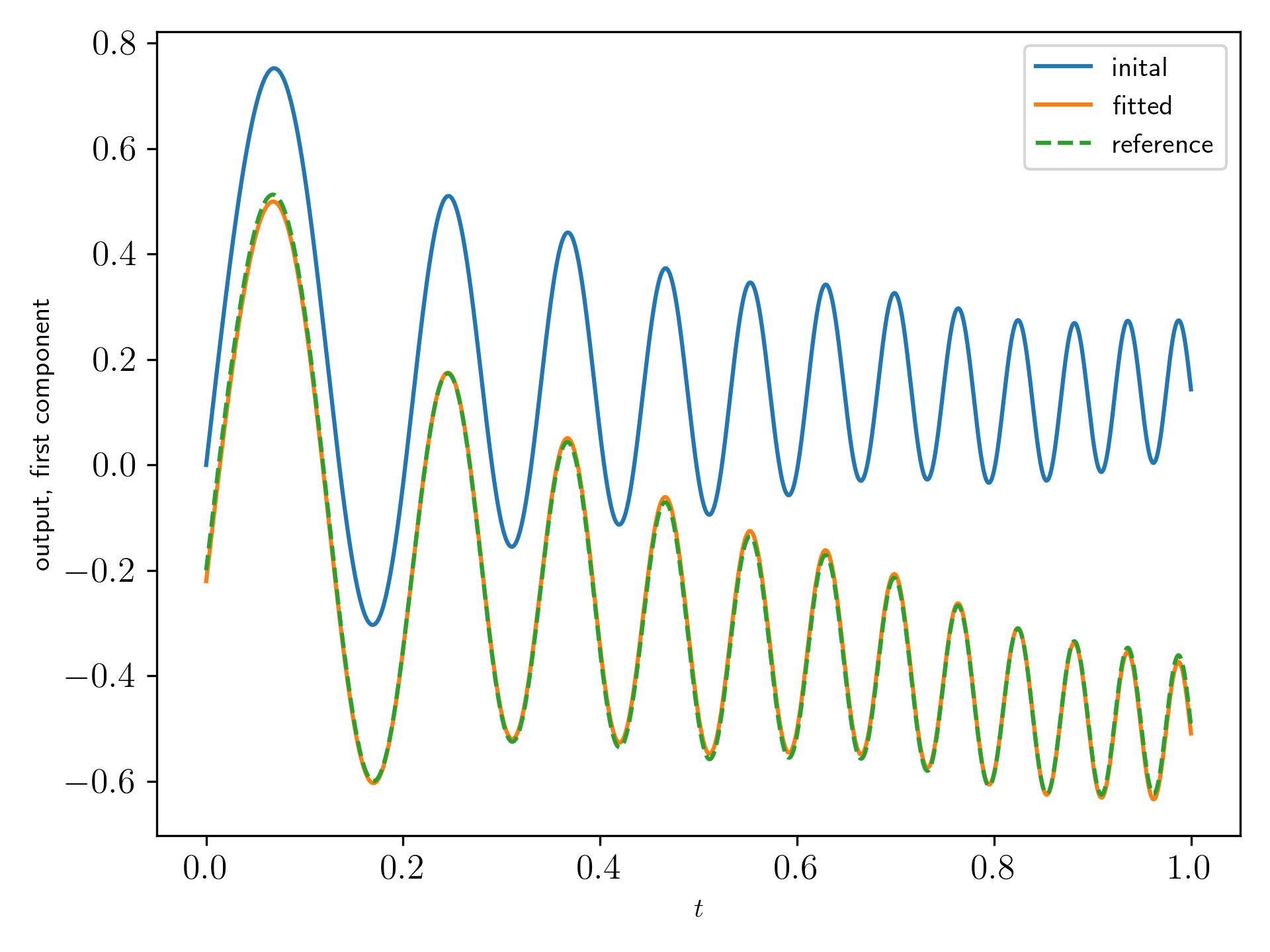}
    \includegraphics[scale=0.4]{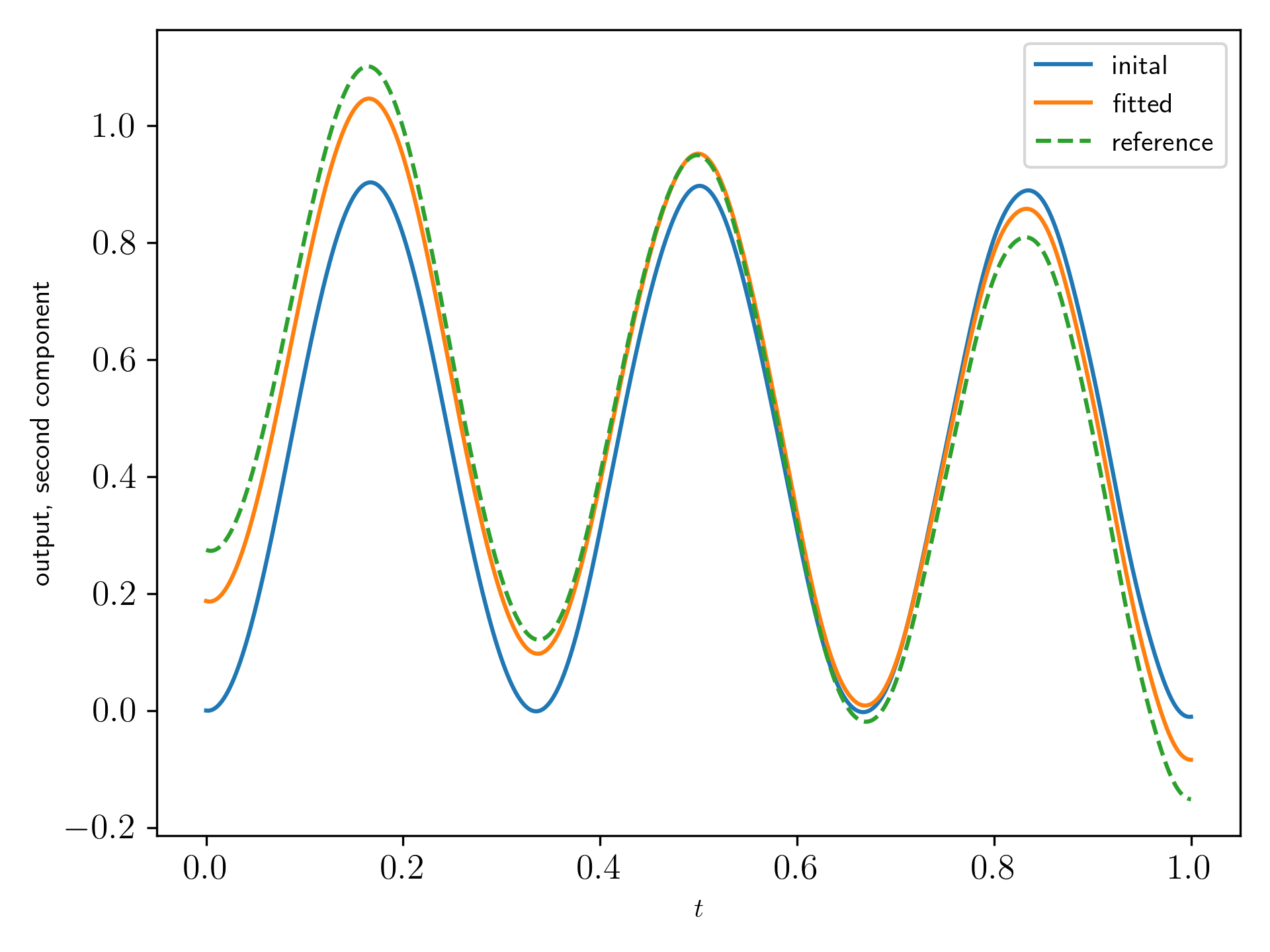}
    \caption{Left: First component of the output for the before the identification (initial), after the identification (fitted) and for the reference matrices (reference) used for the generation of the synthetic data. Right: Second component of the output for the before the identification (initial), after the identification (fitted) and for the reference matrices (reference)used for the generation of the synthetic data.}
    \label{fig:dataoutput_X}
\end{figure}

The algorithm terminated after the maximal number of iterations as can be seen in Figure~\ref{fig:costX}, where we show the evolution of the cost functional on the left-hand side and some additional information of the difference between the fitted and reference output in the maximum norm. Note that we cannot expect higher accuracy, since the error in the cost is of the order of the accuracy of the Euler scheme that we use to solve the state and adjoint systems. The reduction of the cost functional is above $99\%$. 

In Figure~\ref{fig:dataoutput_X} we show the output obtained with the initial matrices before the identification in blue, the orange graph shows the output of the fitted system and the green-dashed line shows the output obtained with the reference data. As the error values above already indicated the fitted output matches the reference output very well.

We close the section on the deterministic data with a cross validation test. The input signal $u_\text{test}$ is now applied to the fitted and reference systems and we compare the corresponding output. Figure~\ref{fig:robustness_det} shows the results of this cross validation graphically. In numbers, we obtain for the L2-difference $\| y - y_\text{test}\|_{L^2((0,1);\mathbb R^2)} = 4.58e^{-2}$ and for maximum norm $\| y - y_\text{test}\|_{\infty} = 8.75e^{-2}$. 
\begin{figure}[ht!]
    \centering
    \includegraphics[scale=0.45]{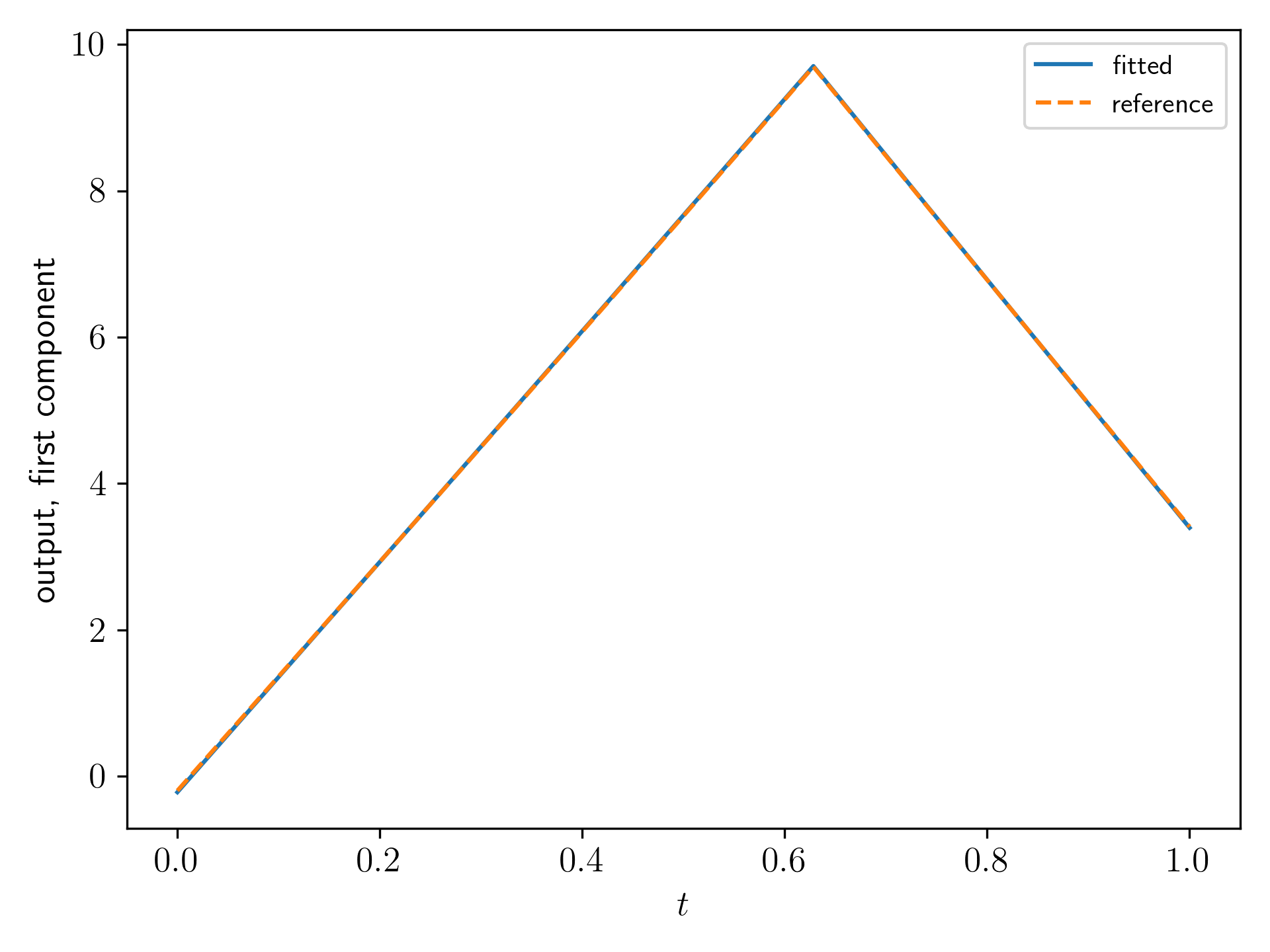}
    \includegraphics[scale=0.45]{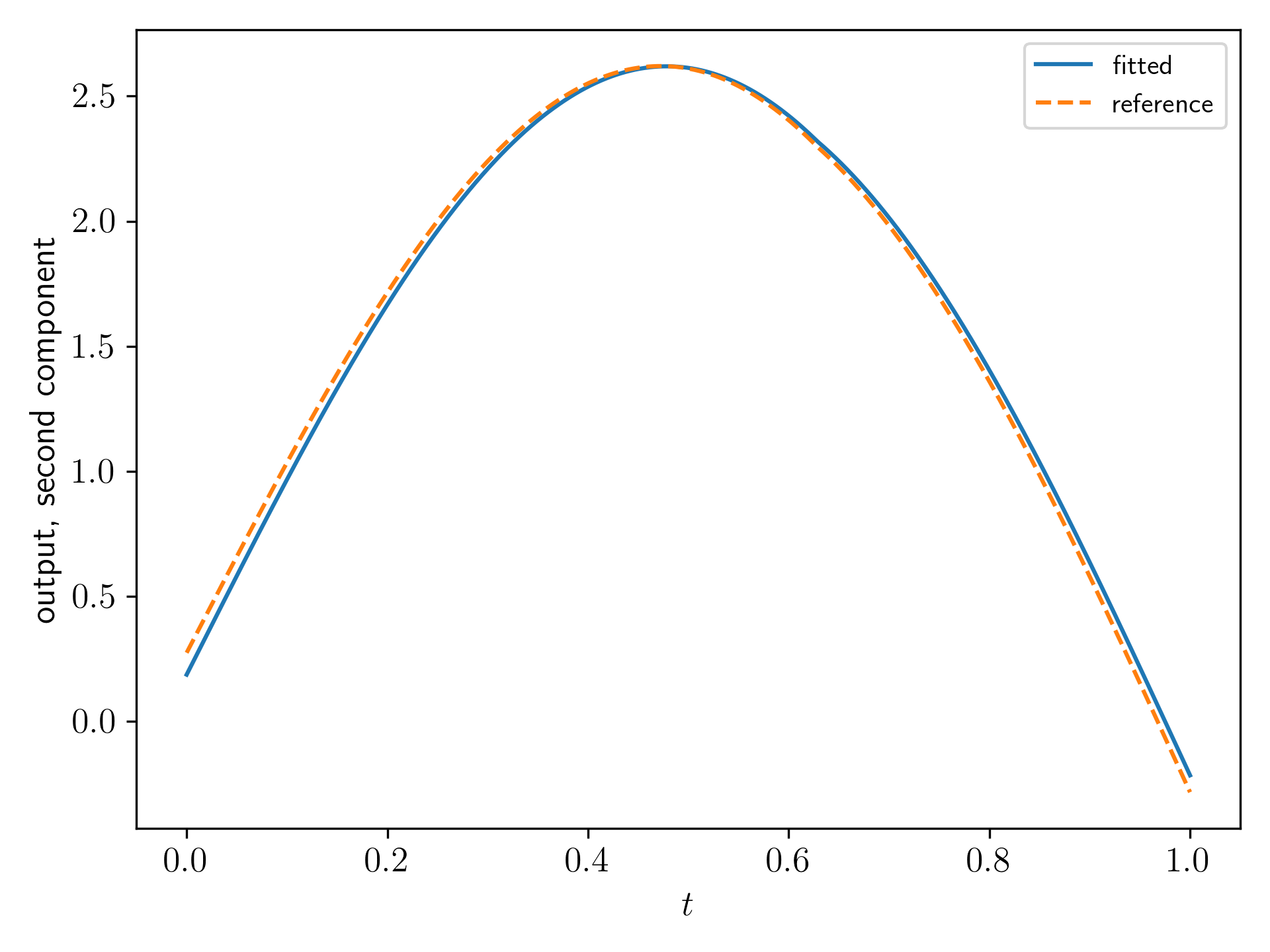}
    \caption{Left: First component of the output obtained in the cross validation. Right: Second component of the output obtained in the cross validation.}
    \label{fig:robustness_X}
\end{figure}

\subsection{Synthetic data (randomly perturbed)}
Let us consider the same setting as in the previous subsection but with randomly perturbed output data. In every time-step we add independent normally distributed vectors $n_i \in \R^m$ leading us to the perturbed data given by
\[
y_\text{data,$\sigma$}(t_i) = y_\textrm{data}(t_i) + \sigma n_i.
\]
We show results for $\sigma\in \{0.01, 0.05, 0.25\}.$ 

\begin{figure}[h!]
    \centering
    \includegraphics[scale=0.6]{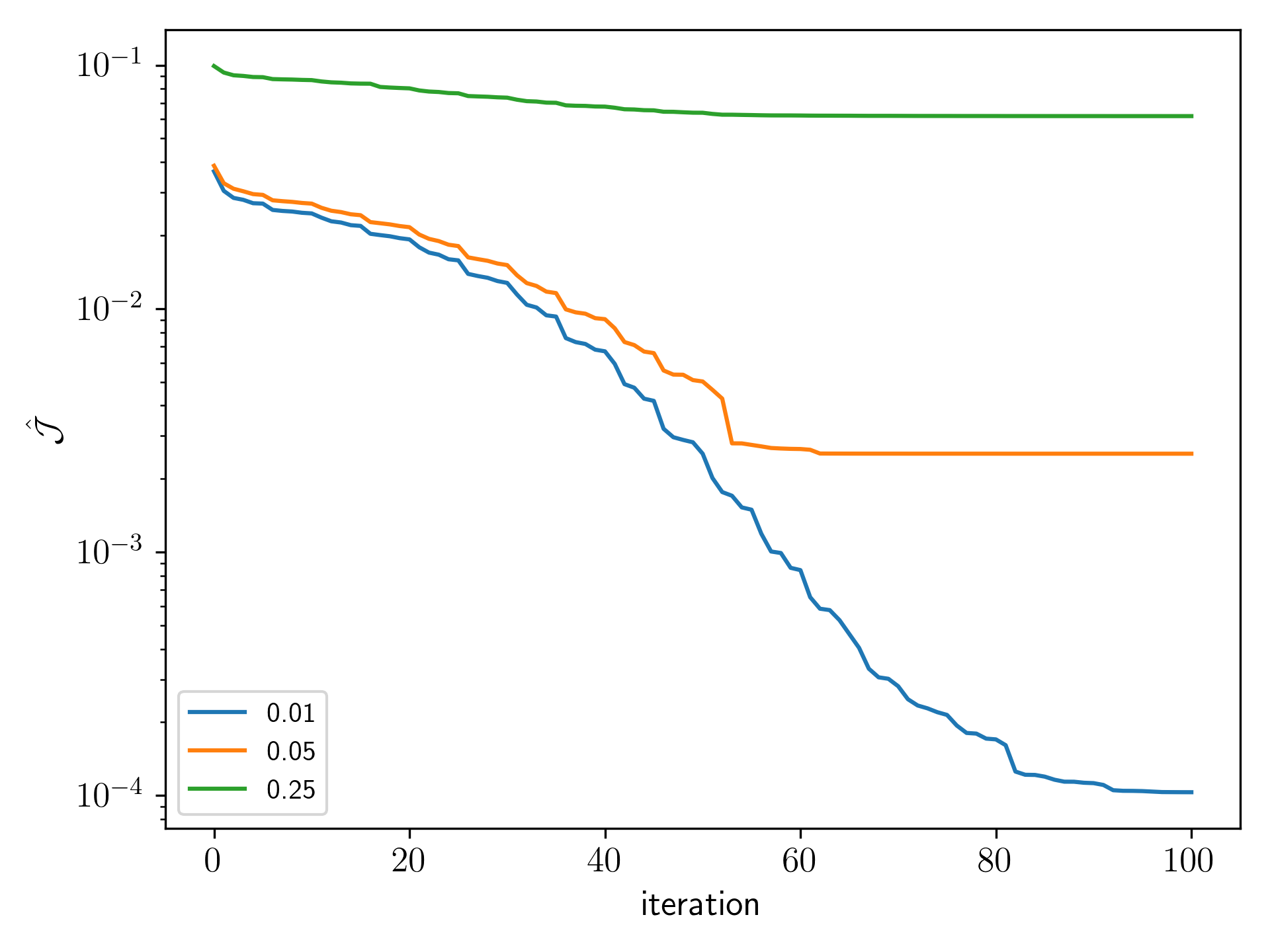}
    \caption{Evolution of the cost functional for difference noise levels. $\sigma \in \{0.01, 0.05, 0.25 \}$.}
    \label{fig:cost_noise}
\end{figure}

\begin{table} \centering
\begin{tabular}{l | c | c | c }
    $\sigma$ & $0.01$ & $0.05$ & $0.25$ \\  \hline
    cost in first iteration & $3.67e^{-2}$ & $3.86e^{-2}$ & $9.92e^{-2}$\\
    cost in last iteration & $1.03e^{-4}$ & $2.53e^{-3}$ & $6.18e^{-2}$\\
    $\|y - y_\text{ref}\|_\infty$ in first iteration & $7.06e^{-1}$ & $7.06e^{-1}$ & $7.06e^{-1}$ \\
     $\|y - y_\text{ref}\|_\infty$ in last iteration & $4.88e^{-3}$ &  $3.66e^{-3}$ & $1.37e^{-2}$ \\
\end{tabular}
\caption{Additional information on the identification with different noise levels.}
\label{tab:error_noise}
\end{table}

Figure~\ref{fig:cost_noise} shows the evolution of the cost functionals over the identification iterations for the different noise levels. For $\sigma=0.25$ the cost is almost constant, hence the noise overlays the information. However, the error values in Table~\ref{tab:error_noise} show that the identification works stable. This is also true for the cross validation shown in Figure~\ref{fig:cross_noise} and the corresponding errors in Table~\ref{tab:cross_error_noise}. It is interesting to see that the robustness in the cross validation test increases with the noise level in the training.

\begin{figure}[ht!]
    \centering
    \includegraphics[scale=0.5]{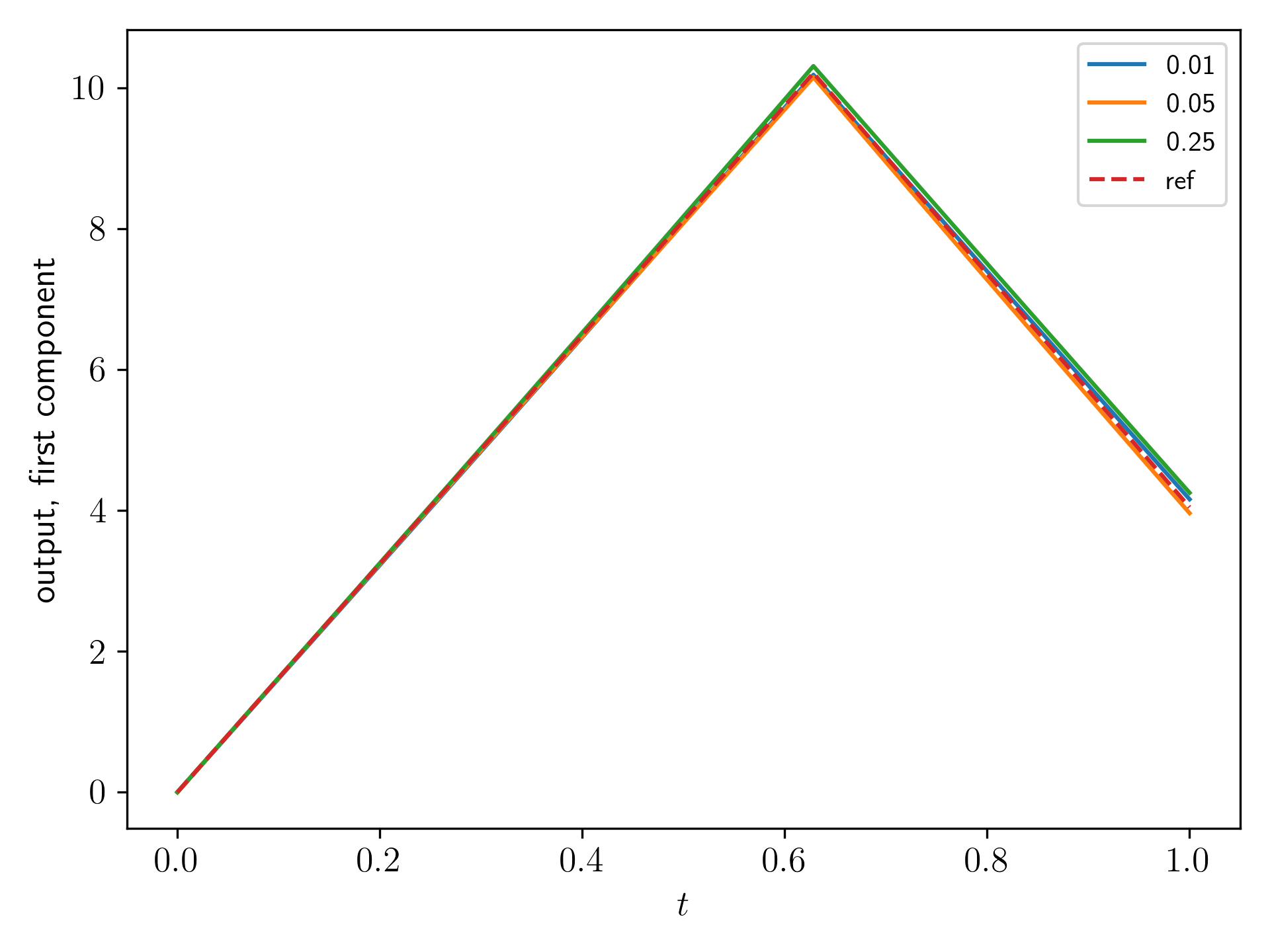}
    \includegraphics[scale=0.5]{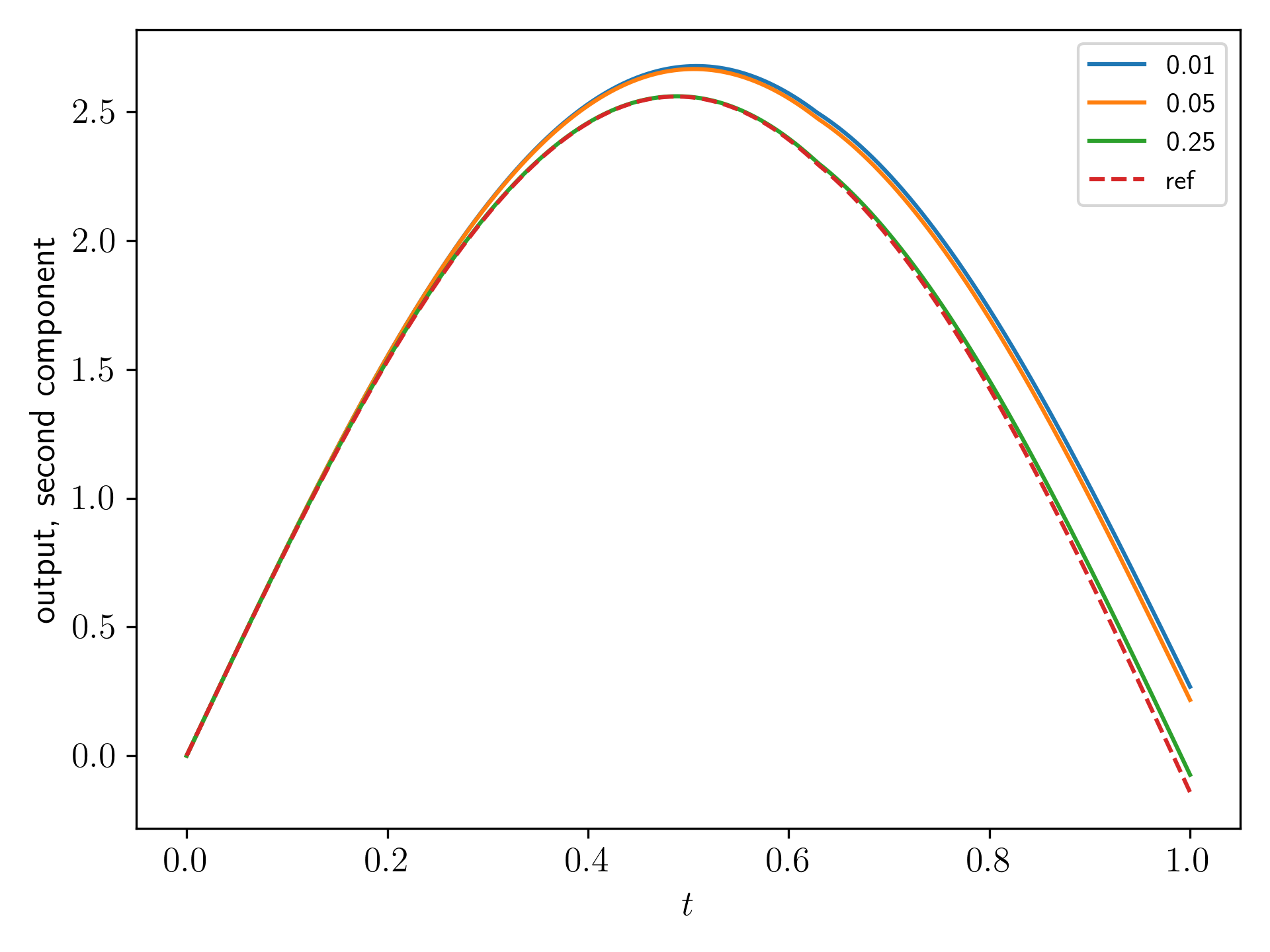}
    \caption{Output of the cross validation test for different noise levels $\sigma \in \{0.01, 0.05, 0.25 \}$.}
    \label{fig:cross_noise}
\end{figure}

\begin{table} \centering
\begin{tabular}{l | c | c | c }
    $\sigma$ & $0.01$ & $0.05$ & $0.25$ \\  \hline
    $\|y - y_\text{test}\|_{L_2((0,1);\mathbb R^2)}$ & $2.09e^{-1}$ & $1.89e^{-1}$ & $1.01e^{-1}$ \\
     $\|y - y_\text{test}\|_\infty$ & $4.1e^{-1}$ &  $3.58e^{-1}$ & $2.04e^{-1}$ \\
\end{tabular}
\caption{Errors obtained in the cross validation for different noise levels.}
\label{tab:cross_error_noise}
\end{table}

\subsection{Model order reduction test with single mass-spring-damper chain}
The last test case we consider is taken from the PHS benchmark collection \cite{PHSbenchmarks} and we aim to investigate the robustness of our approach in case of model order reduction, i.e.~if the model size $n$ deviates for the real data.

We generated the standard example with $50$ mass-spring-damper cells that are each connected with their neighboring masses by springs. The last mass is connected to a wall via a spring while at the two first masses external forces are applied. These are two dimensional inputs leading to $m=2.$ Moreover, each mass is connected with the ground with a damper. All masses are set to $4,$ the damping coefficient to $1$ and the stiffness to $4$. The output $y$ are  the velocities of the masses which are controlled. For an illustration and more details we refer to the documentation of the MSD Chain example referenced at \cite{PHSbenchmarks}. 

\begin{figure}[ht!]
    \centering
    \includegraphics[scale=0.45]{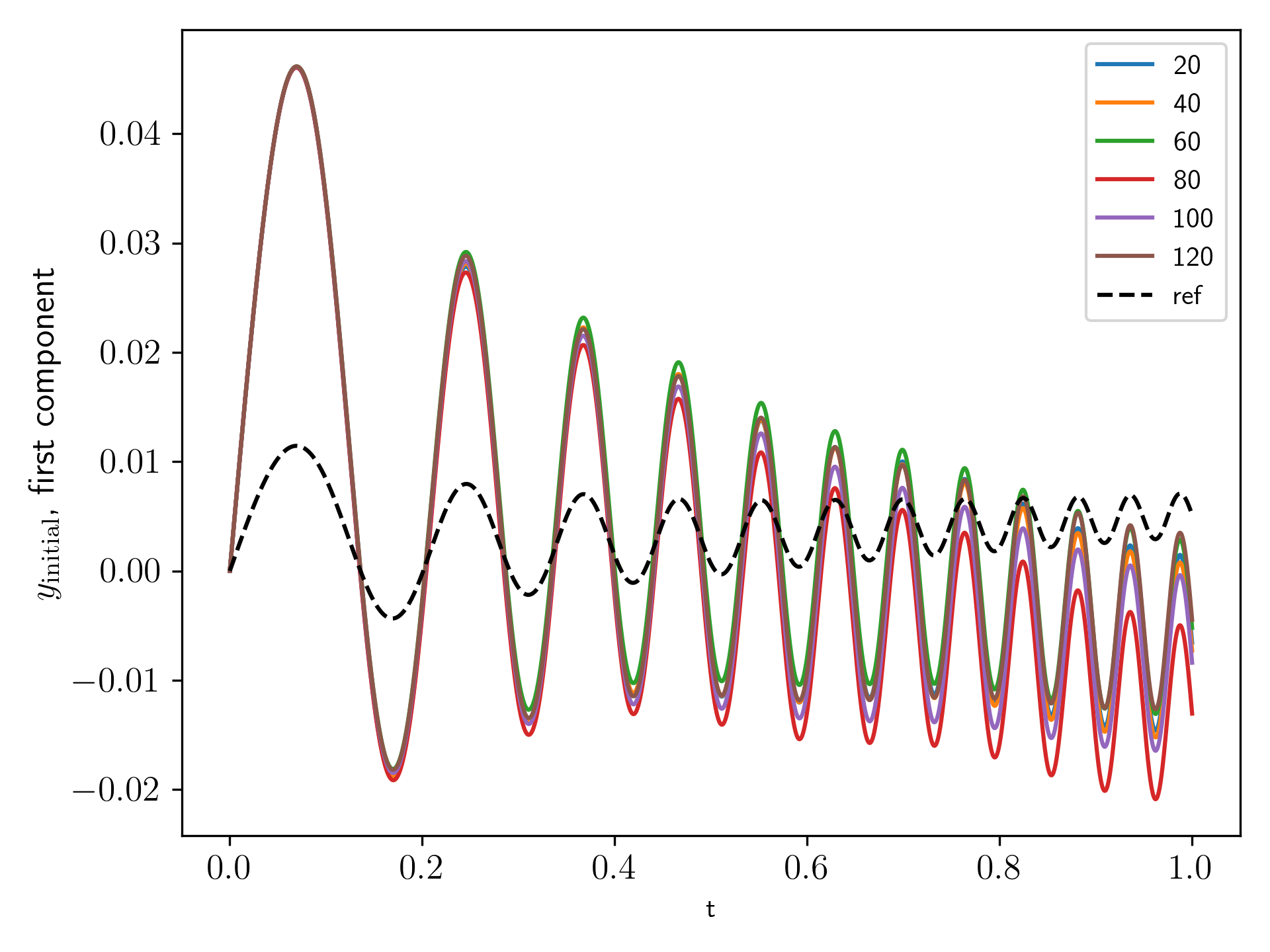}
    \includegraphics[scale=0.45]{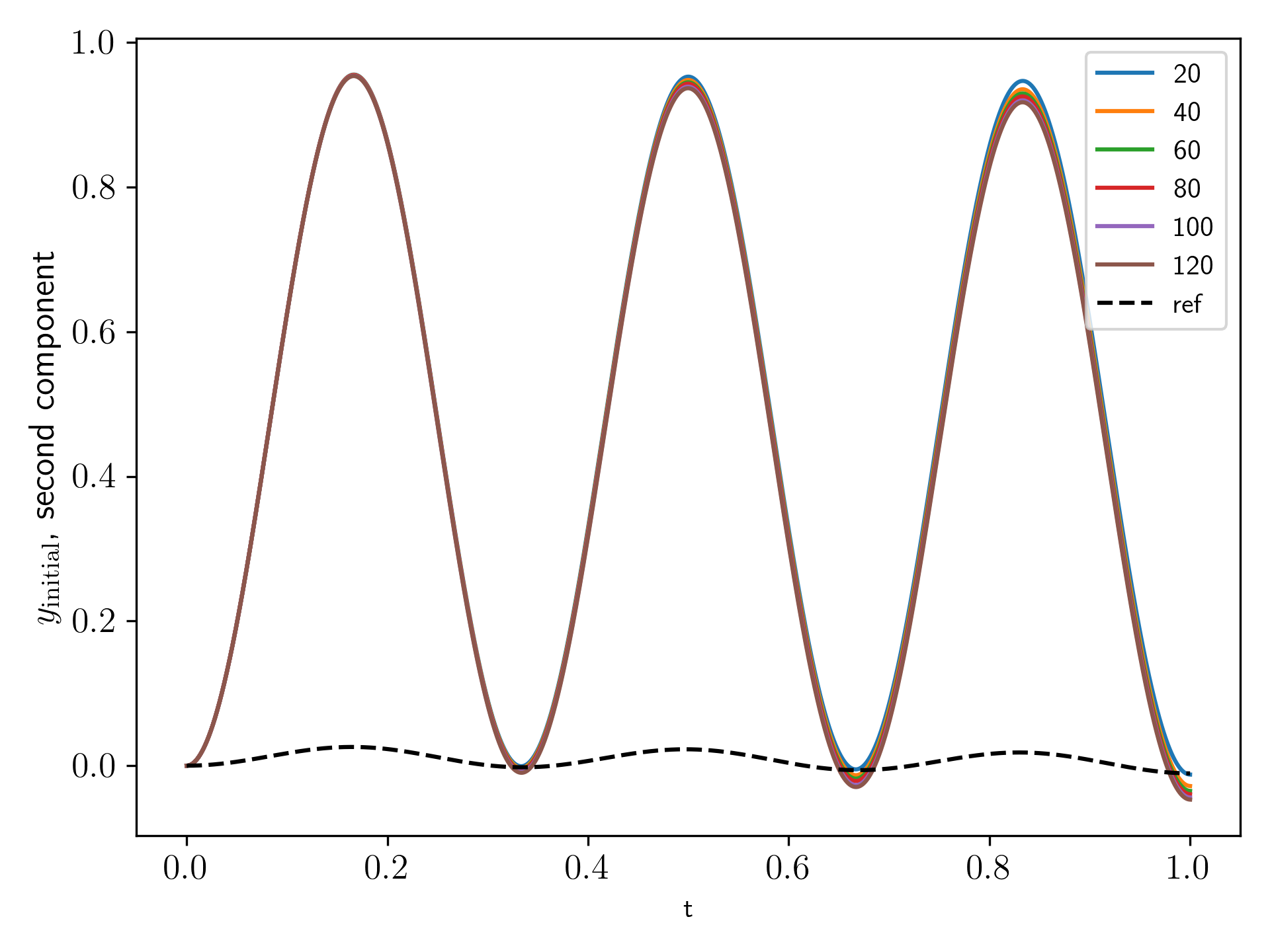}
    \caption{Left: first component of the initial output for different model sizes $n$. Right: second component of the initial output for different model sizes $n$.}
    \label{fig:initial_MSD}
\end{figure}

Figure~\ref{fig:initial_MSD} shows the output for the initial configuration before the identification for different model sizes $n$ and the reference output as black dashed line.  

\begin{figure}[ht!]
    \centering
    \includegraphics[scale=0.45]{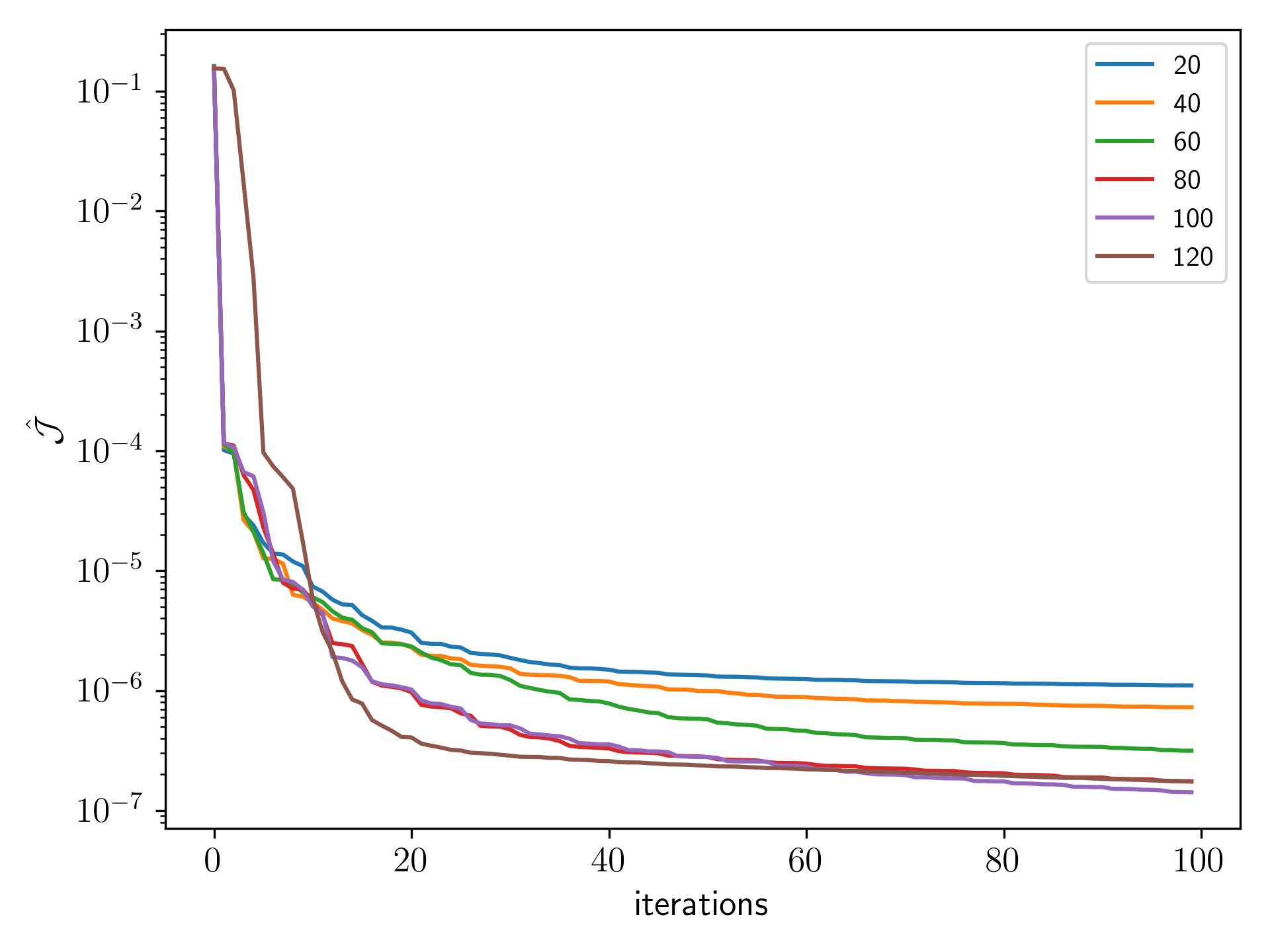}
    \includegraphics[scale=0.45]{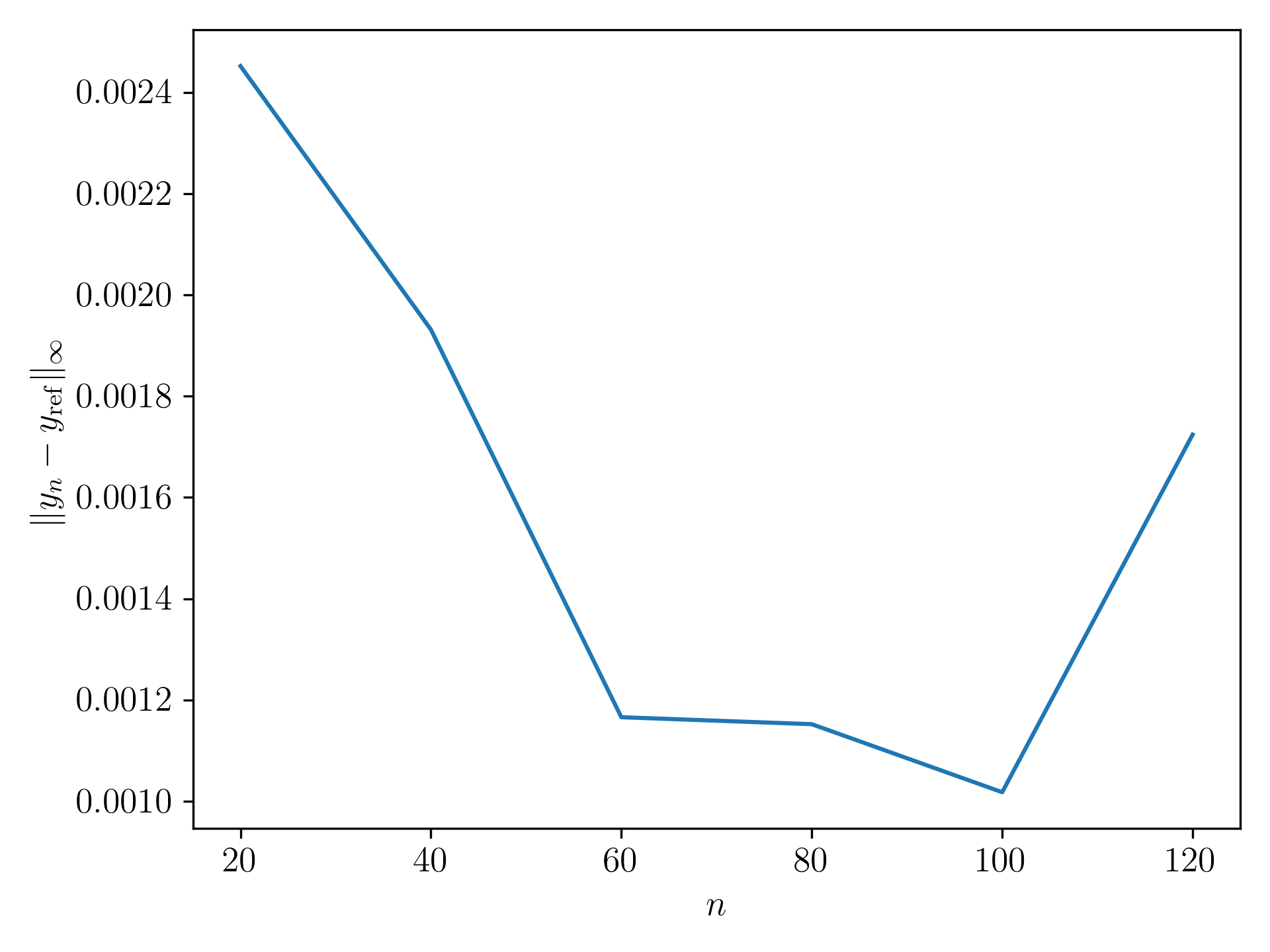}
    \caption{Left: evolution of the cost functional over the identification iterations for different model sizes $n$. Right: Relative differences of the model output $y_n$ and the reference output $y_\text{ref}$ measured in the maximum norm for different model sizes $n$.}
    \label{fig:results_MSD}
\end{figure}

We extracted the system matrices $J,Q,R \in \R^{100 \times 100}$ and the input matrix $B \in \R^{100\times 2}$ from the julia implementation and stored them in python npy-format to use them with our implementation of the algorithm. 

In Figure~\ref{fig:results_MSD} we show the evolution of the cost functional over the identification iterations and the relative error of the output compared in the maximum norm for different model sizes $n$ with true reference size $n=100$. As expected models with larger $n$ are able to capture the dynamics better, however, even for the $n=20$ the error after the identification is below $10$\% and the values for $n=60,80$ and $100$ are below
$5$ \%. While the cost corresponds to the $L_2$-norm, we show the relative error of the output measured in the maximum norm on the left panel in Figure~\ref{fig:results_MSD}. 

\begin{table}[ht!]
\centering
\begin{tabular}{l | r r r r r r }
n & 20 & 40 & 60 & 80 & 100 & 120 \\ \hline  
$\| y_n - y_\mathrm{ref}\|_\infty / \|y_\mathrm{ref}\|_\infty$ before & $359 \%$ & $359 \%$ & $359 \%$ & $359 \%$ & $359 \%$ & $359 \%$ \\
$\| y_n - y_\mathrm{ref}\|_\infty / \|y_\mathrm{ref}\|_\infty$ after & $9.48 \%$ & $7.47 \%$ & $4.51 \%$ & $4.45 \%$ & $3.93 \%$ & $6.66 \%$ \\
\end{tabular}
\caption{Relative errors of the output measured before and after the identification}
\label{tab:rel_error_model_reduction}
\end{table}

While the relative error before the identification is about $359 \%$ for all model sizes the error after the identification ranges between $10-5$\%. The exact values are shown in Table~\ref{tab:rel_error_model_reduction} and the fitted output are shown in Figure~\ref{fig:fitted_MSD}.

\begin{figure}[ht!]
    \centering
    \includegraphics[scale=0.45]{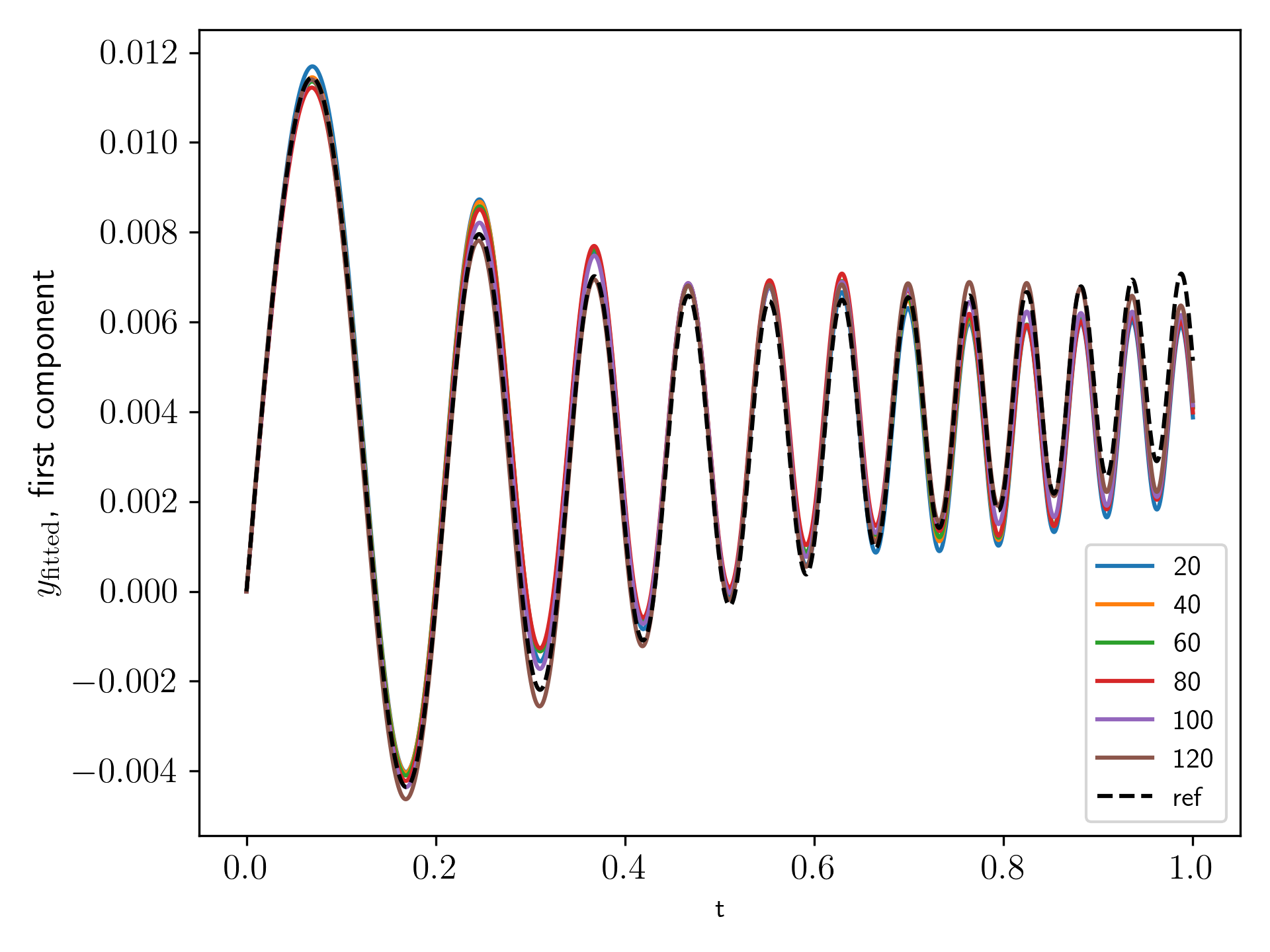}
    \includegraphics[scale=0.45]{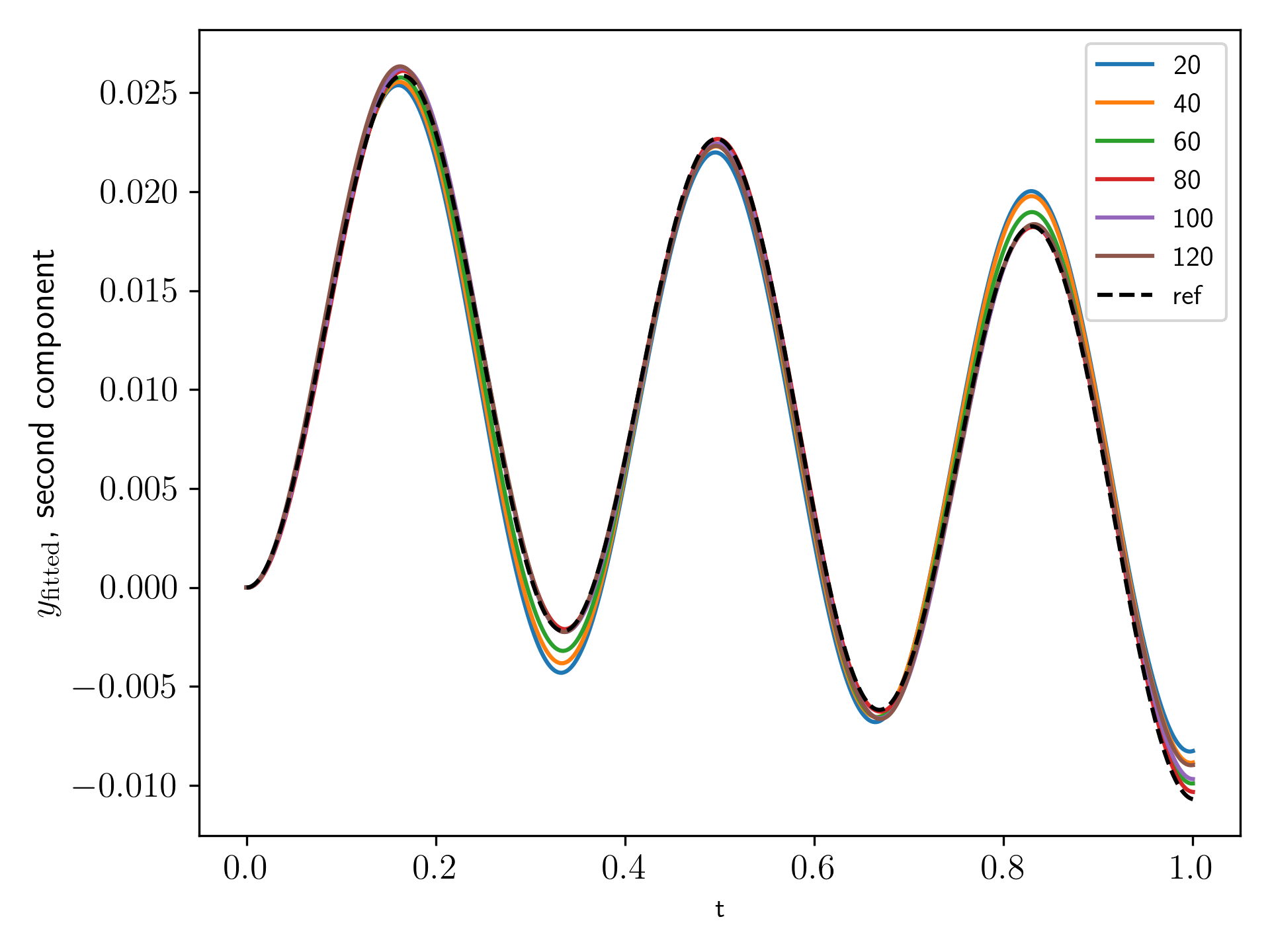}
    \caption{Left: first component of the fitted output for different model sizes $n$. Right: second component of the fitted output for different model sizes $n$.}
    \label{fig:fitted_MSD}
\end{figure}

\section{Conclusion and outlook}
In this paper we have proposed a way of calibrating linear PHS models from given input-output time data in a direct approach: structure preservation defines a constrained optimization problem, which is solved by an adjoint-based conjugate gradient descent algorithm --- there is no need of first deriving a best-fit linear state-space model and then, in a post-processing step, finding the nearest PHS realization. In addition, there is no need for parameterization to transfer the constrained optimization problem into an unconstrained one. Numerical results underpin the feasibility of our approach for both,  synthetic data and a model reduction example employing a mass-spring-damper benchmark problem. 

\section*{Declarations}

\begin{itemize}
\item \textbf{Funding:} Not applicable
\item \textbf{Competing interests:}  The authors declare that they have no competing interests.
\item \textbf{Ethics approval:}
Not applicable
\item \textbf{Availability of data and materials:} The code is publicly available on github: \\
https://github.com/ctotzeck/PHScalibration.
\item \textbf{Authors' contributions:} This whole paper was a joint effort by MG, BJ and CT.  All authors read and approved the final manuscript.

\end{itemize}

\bibliographystyle{alpha}
\bibliography{sample}

\end{document}